\documentclass[12pt,a4paper]{article}
\usepackage[margin=2.8cm,footskip=1cm]{geometry}
\usepackage[utf8]{inputenc}
\usepackage[T1]{fontenc}
\usepackage[english]{babel}
\usepackage{amsmath}
\usepackage{amsfonts}
\usepackage{amssymb}
\usepackage{amsthm}
\usepackage{enumitem}
\usepackage{ucs}
\usepackage{graphicx}
\DeclareGraphicsExtensions{.png,.pdf}
\graphicspath{pictures/}
\usepackage[svgnames]{xcolor}
\usepackage{dsfont}
\usepackage{tikz}
\usepackage{pgfplots}
\usepackage{wasysym}
\usepackage{physics}
\usepackage{mathtools}
\usepackage{xifthen}
\usepackage{url}
\usepackage[hidelinks]{hyperref}
\usepackage{constants}
\usepackage{bm}
\usepackage{authblk}
\usepackage[font=footnotesize]{caption}

\newconstantfamily{n0}{
symbol=\tilde{n},
}

\usepackage{stmaryrd}
\makeatletter
\newcommand{\oset}[3][0ex]{%
  \mathrel{\mathop{#3}\limits^{
    \vbox to#1{\kern-2\ex@
    \hbox{$\scriptstyle#2$}\vss}}}}
\makeatother
\newcommand{\backvec}[1]{{\oset{\shortleftarrow}{#1}}}

\newcommand{\babs}[1]{\bigl\lvert #1 \bigr\rvert}
\newcommand{\Babs}[1]{\Bigl\lvert #1 \Bigr\rvert}

\DeclarePairedDelimiterX{\setof}[2]{\{}{\}}{#1 \,:\, #2}

\newcommand{\given}{\,|\,}
\newcommand{\bgiven}{\,\big|\,}
\newcommand{\Bgiven}{\,\Big|\,}

\newcommand{\Var}{\mathrm{Var}}

\newcommand{\BrownBri}{\mathfrak{b}}

\newcommand{\N}{\mathbb{N}}
\newcommand{\Z}{\mathbb{Z}}
\newcommand{\R}{\mathbb{R}}

\newcommand{\bbC}{\mathbb{C}}
\newcommand{\bbD}{\mathbb{D}}

\newcommand{\bbH}{\mathbb{H}}
\newcommand{\bbI}{\mathbb{I}}
\newcommand{\bbM}{\mathbb{M}}
\newcommand{\bbP}{\mathbb{P}}

\newcommand{\rmd}{\mathrm{d}}
\newcommand{\rmi}{\mathrm{i}}

\newcommand{\calF}{\mathcal{F}}

\newcommand{\calM}{\mathcal{M}}
\newcommand{\calN}{\mathcal{N}}

\newcommand{\ThetaJac}{\Theta_J}

\newcommand{\ControlGrad}{\mathrm{CtrInc}}

\theoremstyle{plain}
\newtheorem{theorem}{Theorem}[section]
\newtheorem{lemma}[theorem]{Lemma}

\newtheorem{remark}{Remark}[section]

\newtheorem{claim}{Claim}

\theoremstyle{definition}

\newtheorem*{definition*}{Definition}

\title{Finite-time trajectorial estimates for inhomogeneous random walks}
\author[$\dagger$]{Sébastien Ott}
\author[$\star$]{Yvan Velenik}
\affil[$\dagger$]{EPFL, Lausanne, CH}
\affil[$\star$]{University of Geneva, Geneva, CH}

\date{\today}

\begin{document}

\maketitle

\begin{abstract}
	We consider integer-valued random walks with independent but not identically distributed increments, and extend to this context several classical estimates, including a local limit theorem, precise small-ball estimates (both conditional on the final point and unconditional), and bounds on the probability that the random walk trajectory remains positive up to a given time (again, both conditional on the final point and unconditional). Two key features of this work are that the bounds are non-asymptotic, holding true for finite time horizons, and, crucially, that the latter hold uniformly over an entire class of admissible increment sequences. This provides a robust framework for applications. These results are, in particular, tailored for the analysis of processes derived through a time-dependent tilting of the increments of a time-homogeneous random walk.
\end{abstract}

\section{Introduction and main results}
\label{sec:intro}

\subsection{Motivations and existing results}
\label{subsec:intro:motiv}

Let \((X_k)_{k\geq 1}\) be independent and identically distributed (i.i.d.) \(\Z\)-valued random variables. Given a starting point \(u\in\Z\), the associated random walk \((S_n)_{n\geq 0}\) is defined by \(S_0 = u\) and \(S_n=S_{n-1}+X_n\) for \(n\geq 1\).

It is often extremely useful across various contexts to consider a tilted version of the increments \((X_k)_{k\geq 1}\), and consequently of the random walk \((S_n)_{n\geq 0}\). Specifically, fix two real numbers \(a<b\) and, for each \(k\geq 1\), let \(t_k\in [a,b]\). A new sequence of random variables \((\widehat{X}_k)_{k\geq 1}\) is then defined by
\[
	\forall \ell\in\N,\qquad
	P(\widehat{X}_k = \ell) = \frac{e^{t_k \ell}}{E(e^{t_k X_k})} \bbP(X_k=\ell),
\]
and the associated random walk is denoted by \((\widehat S_n)_{n\geq 0}\). Let us also denote by \((\bar{S}_n)_{n\geq 0}\) the random walk with centred increments \(\bar{X}_k = \widehat{X}_k - E(\widehat{X}_k)\).

While the new increments \((\widehat{X}_k)_{k\geq 1}\) (or \((\bar{X}_k)_{k\geq 1}\)) remain independent, they are, in general, no longer identically distributed due to the time-dependent tilt \(t_k\). The restriction that the latter parameters belong to the compact interval \([a,b]\subset\R\) ensures a degree of uniformity for the tilted distributions.

The resulting time-inhomogeneity unfortunately precludes the use of various classical estimates that rely on the i.i.d.\ nature of the increments.  These include, for example, estimates on the probability that the walk stays positive for a time \(n\) (possibly conditioned on its position at time \(n\)), as well as small-ball estimates (also possibly conditioned on the position at time \(n\)). We refer to~\cite{Caravenna+Chaumont-2013}, and references therein, for a treatment of the i.i.d.\ case.

Our primary goal in this work is to provide extensions of such classical results that are valid, in particular, for this specific class of time-inhomogeneous random walks.

We emphasize that while similar estimates have been derived for time-in\-ho\-mo\-ge\-neous random walks (in particular in \cite{Denisov+Sakhanenko+Wachtel-2018}), the focus was on obtaining sharp estimates for the asymptotic behavior (i.e., as \(n\to\infty\)) for a \emph{fixed} sequence of increments.

In contrast, we are interested in estimates valid for a fixed, finite time \(n\), and uniform over any admissible sequence of increments. Moreover, we are satisfied with non-sharp bounds -- that is, upper and lower bounds that differ by a multiplicative constant. We are not aware of results of this type in the existing literature.

We have strived to formulate our results in a manner that allows for their easy importation into other research. To maximize their potential usability, we also tried to weaken the moments assumptions far beyond what would be needed to treat the tilted random walks described above. Two direct applications of the results of this paper are presented in~\cite{Ott+Velenik-2025}, which considers a one-dimensional random walk constrained to stay above a macroscopic concave obstacle, and~\cite{Khettabi-2025}, which analyzes an effective model of a polymer hanging in a gravitational field.

\subsection{Main results and roadmap to the paper}
\label{subsec:intro:results}

We present now, in a somewhat informal way, the main results derived in this work, with references to the precise statements that can be found later in the paper.

Let us denote by \((S_n)_{n\geq 0}\) the time-inhomogeneous random walk, with independent increments \((X_k)_{k\geq 1}\); we shall write \(P_u(\cdot) = P(\cdot \,|\, S_0=u)\). A complete description of our general setup and assumptions is provided in Section~\ref{sec:notations}. The main assumption is a form of uniformity of the increments distributions.

\subsubsection{Local limit theorem}
\label{sssec:Results:LLT}

The first result, which is also an important tool in deriving the ``bridge'' version of some of our other claims, is a version of the local limit theorem for a time-inhomogenous \(\Z\)-valued random walk. It applies in the whole Gaussian regime, i.e., it provides sharp Gaussian approximation for \(P_0(S_n=y)\) for any \(y\) satisfying \(|y-E_0(S_n)| \leq n^\alpha\) with \(\alpha < 2/3\). Namely, it is proved in Theorem~\ref{thm:inhomo_LLT} that, under suitable assumptions,
\[
	\exp(-Cn^{-\min(2-3\alpha,1/3)})
	\leq
	\sqrt{2\pi B_{n}}e^{(y-m_{n})^2/2B_{n}} P_0(S_{n} = y)
	\leq
	\exp(Cn^{-\min(2-3\alpha,1/3)}),
\]
for some \(C\geq 0\), where \(m_n = E_0(S_n)\) and \(B_n = \Var(S_n) = \sum_{i=1}^n \Var(X_i)\).

\subsubsection{Probability that the walk remains positive}
\label{sssec:Results:Positivity}

Let \(\tau = \min\{k\geq 1 \,:\, S_k<0\}\). We are interested in the probability that \(\tau > n\), for a fixed \(n\), as a function of the starting point and possibly, conditionally on \(S_n\).

\smallskip
Let us start with the case of a free endpoint. Let \(A>0, A'>0, s>0\) and consider a sequence of increments satisfying \(E(X_i) = 0\), \(E(X_i^4)\leq A\) and \(E(X_i\mathds{1}_{X_i>0})\geq s\) for all \(i=1,\dots,n\). Then, there exist \(c_-=c_-(A,A',s)>0\) and \(c_+=c_+(A,s)>0\) such that
\[
	\forall 0\leq u\leq A'\sqrt{n},\qquad
	P_u(\tau>n) \geq \frac{c_-(u+1)}{\sqrt{n}},
\]
and
\[
	\forall u\geq 0,\qquad
	P_u(\tau>n) \leq \frac{c_+(u+1)}{\sqrt{n}}.
\]
The lower bound is proved in Lemma~\ref{lem:positivity_general_walks:LB} and the upper bound in Lemma~\ref{lem:positivity_general_walks:UB}.

\smallskip
Let us now turn to the ``bridge'' case. The estimates in this case assume that the increments satisfy the ``uniformity'' assumption detailed in Section~\ref{sec:notations}. Let \(\alpha\in (1/2,2/3)\).
There exist \(C_-,C_+,c_-,c_+\in (0,+\infty)\) and \(n_0\geq 1\) such that, for any \(n\geq n_0\) and any \(0\leq u,v\leq n^{\alpha}\) with \(P_u(\bar{S}_n = v)>0\),
\[
	P_u\bigl(\bar{S}_n = v,\ \min_{i=1,\dots,n} \bar{S}_i \geq 0 \bigr)
	\geq
	\frac{C_- \min(u+1,\sqrt{n})\min(v+1,\sqrt{n})}{n^{3/2}}e^{-c_-(u-v)^2/n},
\]
and
\[
	P_u\bigl(\bar{S}_n = v,\ \min_{i=1,\dots,n} \bar{S}_i \geq 0 \bigr)
	\leq
	\frac{C_+ \min(u+1,\sqrt{n})\min(v+1,\sqrt{n})}{n^{3/2}}e^{-c_+(u-v)^2/n}.
\]
The lower bound is proved in Lemma~\ref{lem:positivity_fixed_endpoint:LB} and the upper bound in Lemma~\ref{lem:positivity_fixed_endpoint:UB}.

\subsubsection{Small-ball estimates}
\label{sssec:Results:SmallBall}

Let us now turn to small-ball estimates, i.e., the probability that the random walk trajectory remains in the interior of a given ``tube'' for a time \(n\).
As above we treat separately the case of a free endpoint, and that of a bridge.

\smallskip
Let us start with the case of a free endpoint. Let \(A>0\), \(\sigma>0\) and consider a sequence of increments satisfying \(E(X_i) = 0\), \(E(|X_i|^3)\leq A\) and \(E(X_i^2) \geq \sigma^2\) for all \(i=1,\dots,n\). Then, there exist \(c_-,c_+>0\), \(\lambda_0\geq 0\) and \(n_0\geq 1\) such that, for any \(n\geq n_0\),
\[
	\forall \lambda\geq\lambda_0,\qquad
	P_0\bigl(\max_{i=1,\dots,n} |S_i| \leq \lambda,\ |S_n|\leq \lambda/2 \bigr) \geq e^{-c_- n/\lambda^2},
\]
and
\[
	\forall \sqrt{n} \geq \lambda\geq \lambda_0,\qquad
	P_0\bigl(\max_{i=1,\dots,n} |S_i| \leq \lambda \bigr) \leq e^{-c_+ n/\lambda^2}.
\]
The lower bound is proved in Lemma~\ref{lem:micro_small_ball_walk_LB} and the upper bound in Lemma~\ref{lem:micro_small_ball_walk_UB}.

\smallskip
Let us now turn to the case of a bridge. As before, we require that increments to satisfy the ``uniformity'' assumptions described in Section~\ref{sec:notations}.
Let \(\epsilon>0\). There exist \(c_-,c_+,C_-,C_+\in (0,+\infty)\), \(\lambda_0\geq 0\) and \(n_0\geq 1\) such that, for any \(n\geq n_0\) and \(\sqrt{n}\geq \lambda\geq \lambda_0\):
\begin{itemize}
	\item For any \(x\in\Z\) with \(|x-m_n|\leq (1-\epsilon)\lambda\),
		\[
			P_0\bigl(\max_{i=1,\dots,n} |S_i - E(S_i)| \leq \lambda,\; S_n = x \bigr)
			\geq 
			\tfrac{C_-}{\lambda} e^{-c_- n/\lambda^2}.
		\]
	\item For any \(x\in\Z\) with \(|x-m_n|\leq \lambda\),
		\[
			P_0\bigl(\max_{i=1,\dots,n} |S_i - E(S_i)| \leq \lambda,\; S_n = x \bigr)
			\leq 
			\tfrac{C_+}{\lambda} e^{-c_+ n/\lambda^2}.
		\]
\end{itemize}
Both bounds are established in Theorem~\ref{thm:small_ball_fixed_endpoint}.

\smallskip
The next results provide better control on the effect of the starting and ending points. They hold under the same ``uniformity'' assumptions.
Let \(K>0\). There exist \(C_-,C_+,c_-,c_+\in (0,+\infty)\) and \(\lambda_0,n_0\geq 0\) such that the following holds. For any \(n\geq n_0\) any \(\lambda_0\leq \lambda \leq K\sqrt{n}\) and any \(0\leq u,v \leq \lambda\) with \(P_u(\bar S_n = v) > 0\),
\begin{multline*}
	P_u\bigl( \forall i\in\{1,\dots,n\},\; 0\leq \bar S_i \leq \lambda,\; \bar S_n = v \bigr) \\
    \geq
    \frac{C_-(\min(u, \lambda-u)+1)(\min(v, \lambda-v)+1)}{\lambda^3} e^{-c_-n/\lambda^2},
\end{multline*}
and
\begin{multline*}
    P_u\bigl( \forall i\in\{1,\dots,n\},\; 0\leq \bar S_i \leq \lambda,\; \bar S_n = v \bigr)
    \\
    \leq
    \frac{C_+(\min(u, \lambda-u)+1)(\min(v, \lambda-v)+1)}{\lambda^3} e^{-c_+n/\lambda^2}.
\end{multline*}
The lower bound is proved in Lemma~\ref{lem:excursions:small_ball:LB} and the upper bound in Lemma~\ref{lem:excursions:small_ball:UB}.

\smallskip
Finally, we consider the complementary case in which the ``tube'' has a width larger than \(\sqrt{n}\), still under the same ``uniformity'' assumptions.
Let \(\alpha \in (0, 2/3)\). There are \(C_-,C_+,c_-,c_+\in (0,+\infty)\), \(n_0\geq 0\), such that the following holds. For any \(n\geq n_0\), any \(\lambda \geq \sqrt{n}\), and any \(0\leq u,v \leq \lambda\) with \(P_u(\bar{S}_n = v)>0\) and \(|u-v|\leq n^{\alpha}\),
\begin{multline*}
	P_u\bigl( \forall i\in\{1,\dots,n\},\; 0\leq \bar S_i \leq \lambda,\; \bar{S}_n = v \bigr)
	\\
	\geq
	\frac{C_-(\min(u, \lambda-u, \sqrt{n})+1)(\min(v, \lambda-v, \sqrt{n})+1)}{n^{3/2}} e^{-c_-(u-v)^2/n},
\end{multline*}
and
\begin{multline*}
	P_u\bigl(  \forall i\in\{1,\dots,n\},\; 0\leq \bar S_i \leq \lambda,\; \bar{S}_n = v \bigr)	\\
	\leq
	\frac{C_+(\min(u, \lambda-u, \sqrt{n})+1)(\min(v, \lambda-v, \sqrt{n})+1)}{n^{3/2}} e^{-c_+(u-v)^2/n}.
\end{multline*}
Both bounds are established in Theorem~\ref{thm:excursions:ceiling}.

\subsubsection{Tail estimates}
\label{sssec:Results:Tails}

The last estimates address the tail of the one-time marginal of a bridge. Their derivation assumes again the ``uniformity'' assumptions of Section~\ref{sec:notations}.
Let \(\beta\in (0,1/6)\). There exist \(n_0,t_0\geq 0\) and \(C_-,C_+,c_-,c_+\in (0,+\infty)\) such that the following holds. For any \(n\geq n_0\) and any \(0\leq u,v \leq \tfrac12 t\sqrt{n}\) with \(P_u(\bar{S}_n = v)>0\),
\[
	P_u\bigl( \min_{i=1,\dots,n} \bar{S}_i\geq 0,\ \bar{S}_{k}\geq t\sqrt{n},\ \bar{S}_n = v\bigr) \geq \frac{C_-\min(u+1,\sqrt{n})\min(v+1,\sqrt{n})}{tn^{3/2}} e^{-c_- t^2},
\]
and
\[
	P_u\bigl(\min_{i=1,\dots,n} \bar{S}_i\geq 0,\ \bar{S}_{k}\geq t\sqrt{n},\ \bar{S}_n = v\bigr) \leq \frac{C_+\min(u+1,\sqrt{n})\min(v+1,\sqrt{n})}{tn^{3/2}} e^{-c_+ t^2},
\]
for all \(n/3\leq k\leq 2n/3\), and all \(t_0\leq t\leq n^{\beta}\).
Both bounds are proved in Lemma~\ref{lem:tails}.

\subsection{Some remarks and open questions}
\label{subsec:intro:open}

We collect here various observations and remarks about our setup, hypotheses, and results.

One can first wonder about the optimality of the hypotheses for the positivity results described in Section~\ref{sssec:Results:LLT}.
First, we have that our hypotheses
\begin{equation*}
    E(X) = 0,\quad E(X\mathds{1}_{X>0})\geq s, \quad E(X^4) \leq A,
\end{equation*}
for some \(s,A\in (0,+\infty)\), could be replaced by asking instead
\begin{equation*}
    E(X) = 0,\quad E(X^2) \geq a, \quad E(X^4) \leq A,
\end{equation*}
for some \(a,A\in (0,+\infty)\). This is the content of Lemma~\ref{lem:equivalent_second_moment_cond}.

\begin{lemma}
	\label{lem:equivalent_second_moment_cond}
	Let \(A>0\). Let \(X\) be a real random variable with \(E(X^4)\leq A\), \(E(X)=0\). Then,
	\begin{itemize}
		\item for any \(s>0\), \(E(X\mathds{1}_{X>0}) \geq s\) implies \(E(X^2) \geq 4s^2\);
		\item for any \(\sigma>0\), \(E(X^2) \geq \sigma^2\) implies \(E(X\mathds{1}_{X>0}) \geq \frac{\sigma^3}{4\sqrt{2A}}\).
	\end{itemize}
\end{lemma}
\begin{proof}
	First, note that as \(E(X)=0\), \(E(X\mathds{1}_{X>0}) = - E(X\mathds{1}_{X<0})\) so
	\begin{equation*}
		E(|X|) = 2E(X\mathds{1}_{X>0}).
	\end{equation*}
	Suppose first that \(E(X\mathds{1}_{X>0})\geq s>0\). Then, by Jensen inequality
	\begin{equation*}
		E(X^2) \geq E(|X|)^2 \geq 4s^2.
	\end{equation*}
	Suppose then that \(E(X^2) \geq \sigma^2 >0\). We have for \(K>0\),
	\begin{multline*}
		E(|X|)
		\geq
		E(|X|\mathds{1}_{|X|\leq K})
		\geq
		\tfrac{1}{K}E(X^2\mathds{1}_{|X|\leq K})
		=
		\tfrac{1}{K}\bigl(E(X^2) - E(X^2\mathds{1}_{|X|> K})\bigr)
		\\
		\geq
		\tfrac{1}{K}\bigl(\sigma^2 - E(X^4)^{1/2} P(|X|> K)^{1/2}\bigr)
		\geq
		\tfrac{1}{K}\bigl(\sigma^2 - \sqrt{A} \tfrac{\sqrt{A}}{K^2}\bigr).
	\end{multline*}
	Choosing \(K=\sqrt{2A}/\sigma\), we get \(E(|X|) \geq \frac{\sigma^3}{2\sqrt{2A}}\) and thus \(E(X\mathds{1}_{X>0}) \geq \frac{\sigma^3}{4\sqrt{2A}}\).
\end{proof}
Them, looking at the proofs of Lemmas~\ref{lem:positivity_general_walks:LB},~\ref{lem:positivity_general_walks:UB}, one can notice that the same proofs will work under a \(3+\epsilon\) uniform moment condition for any fixed \(\epsilon>0\). We believe (but did not check carefully) that mild adaptations of our arguments will allow to handle the case of a uniform third moment. We however have no idea how to treat lower uniform moments, and extending
the results described in Section~\ref{sssec:Results:Positivity}
to this case will probably requires new ideas. Obtaining optimal moment condition seems to be a nice open problem. The ``obvious'' tentative of asking a second moment uniformly bounded away from \(0\) and \(+\infty\) fails as is shown in Remark~\ref{rem:unif_second_moment_not_enough}.

\begin{remark}
\label{rem:unif_second_moment_not_enough}
    One could think that a condition of the form \(E(X) = 0\) and \(\sigma_-^2\leq E(X^2)\leq \sigma_+^2\) for some \(\sigma_-,\sigma_+\in (0,+\infty)\) would be the optimal condition. It is not the case as the following example demonstrates. Let \(X_1,X_2,\dots\) be an independent sequence of random variables with law
    \begin{equation*}
        P(X_i = i+1) = P(X_i = -i-1) = \frac{1}{2(i+1)^2},
        \quad
        P(X_i = 0) = 1-\tfrac{1}{(i+1)^2}.
    \end{equation*}
    These variables satisfy \(E(X_i) = 0\), and \(E(X_i^2) = 1\). But, for any \(u>0\),
    \begin{equation*}
        P_u\bigl(\min_{i=1,\dots n} S_i \geq 0\bigr)
        \geq
        P_u\bigl(\cap_{i=1}^n \{X_i= 0\}\bigr)
        \geq
        \exp\Bigl(-\tfrac{4}{3}\sum_{i=1}^{\infty} i^{-2} \Bigr)
        =
        \exp(-2 \pi^2/9),
    \end{equation*}which is uniformly lower bounded over \(n\). One should thus (at least) either ask for a uniform upper bound on strictly more than \(2\) moments, or a lower bound on strictly less than \(2\) moments.
\end{remark}

\begin{remark}
The point-wise estimates that we prove in Sections~\ref{sec:small_ball_lattice_bridges},~\ref{sec:fixed_endpoint}, and~\ref{sec:excursions} rely on an inhomogeneous version of the Local Limit Theorem (Theorem~\ref{thm:inhomo_LLT}). We work under an exponential moment condition as our motivation comes from walks with exponential moments. The results/proofs relying on the LLT extend to ``only'' a fourth moment condition and the ``uniform aperiodicity condition'' once one adds the restriction that the endpoint is in the CLT (\(O(\sqrt{n})\)) regime.
\end{remark}

\section{Setup, notations, conventions}
\label{sec:notations}

\subsection{General notations}
\label{subsec:notations:general_not}

If \((a_{i})_{i\in I}\), \((b_{i})_{i\in I}\) are two indexed collections of real numbers, we will write \(a\leq b\) for \(a_i\leq b_i\) for all \(i\in I\). We will often leave the choice of \(I\) implicit in the notation when there is no risk of confusion. For \(c\in \R\), we also write \(a\geq c\) as a shorthand for \(a_i\geq c \ \forall i\in I\). 

\subsection{Probability measures and walks}
\label{subsec:notations:proba_walks}

We will work on some abstract probability space \((\Omega,\calF,P)\), and suppose that it is large enough to contain all the needed variables.

\subsubsection*{General walks}

For definiteness, \(S_0\) will be a \(\calN(0,1)\) random variable, independent of everything else. The particular distribution of \(S_0\) will not play any role, as we will always condition on its value. For a sequence of random variables \(X_1,X_2,\dots\), we define
\begin{equation*}
	S_{k,n} = \sum_{i=k}^{n} X_i,
	\quad
	\bar{X}_i = X_i-E(X_i),
	\quad \bar{S}_{k,n} = \sum_{i=k}^{n} \bar{X}_i,
	\ 1\leq k \leq n.
\end{equation*}
We also will often use the notations
\begin{equation*}
	m_i = E(S_{1,i}),\quad B_i = \Var(S_{1,i}).
\end{equation*}
We also define the walk started at \(S_0\):
\begin{equation*}
	S_n = S_0 + S_{1,n},
	\quad
	\bar{S}_n \equiv S_0 + \bar{S}_{1,n},
\end{equation*}
and the first return time to the lower half space
\begin{equation*}
	\tau = \min\{k\geq 1: \ S_k<0\}.
\end{equation*}
Finally, we will use the notation
\begin{equation*}
	P_u \equiv P(\cdot \given S_0=u).
\end{equation*}

\subsubsection*{Integer-valued walks}

We say that a \(\Z\)-valued random variable \(X\) is \emph{aperiodic} and/or \emph{irreducible} if the random walk with i.i.d.\ steps having the law of \(X\) is.
Denote \(\calM_{\Z}\) the set of probability measures on \(\Z\) identified with the set or sequences in \([0,1]^{\Z}\) summing to \(1\). We say that a probability measure \(p\in \calM_{\Z}\) is aperiodic and/or irreducible if a random variable having law \(p\) is.
We shall also say that a sequence \(a\in [0,1]^{\Z}\) is aperiodic and/or irreducible if there is an aperiodic/irreducible probability measure \(p\in \calM_{\Z}\), and positive numbers \(b_i>0, i\in \Z\) such that \(a_i = b_ip(i)\).
For \(\delta,c>0\), \(a\in [0,1]^{\Z}\), define
\begin{equation*}
    \calM_{\delta,c}^a = \setof[\Big]{ p\in \calM_{\Z}}{\sup_{\abs{t}\leq \delta}\sum_{n\in \Z}p(n)e^{t n}\leq c,\ p(i) \geq a_i \ \forall i\in \Z }.
\end{equation*}
Note that for any \(a,\delta,c\), \(\calM_{\delta,c}^a\) is compact. Moreover, if \(a\) is irreducible and/or aperiodic, so are every \(p\in \calM_{\delta,c}^a\). Note that when \(a\) is irreducible,
\begin{equation*}
    0< \sigma_-(a,\delta,c)\coloneqq \inf_{p\in \calM_{\delta,c}^a} \Var_p(X) \leq \sup_{p\in \calM_{\delta,c}^a} \Var_p(X) \eqqcolon \sigma_+(a,\delta,c) <\infty
\end{equation*}
where \(X\sim p\).
For \(p\in \calM_{\Z}\) with some exponential moments, we will denote
\begin{gather*}
    M_p(z) = \sum_{k\in \Z}p(k) e^{z k},\quad
    H_p(z) = \ln(M_p(z)),\\
    \mu_p = E_p(X) = H_p'(0),\quad
    \sigma_p^2 = \Var_p(X)=H_p''(0),
\end{gather*}
where \(X\sim p\).

\subsection{Constants}
\label{subsec:notations:constants}

Numbered constants of the form \(c_1,c_2,C_1,C_2,\dots\) have a value that is fixed throughout the paper, while the use of non-numbered constants like \(c,C,c',C',\dots\) will be allowed to vary from line to line in the course of a proof, and are depending only on parameters which values are fixed in the concerned theorem/lemma.

\subsection{Some generalities about walks}
\label{subsection:mart_submart_FKG}

We collect a few standard facts about random walks that we will use repeatedly.

\begin{enumerate}
    \item If \(X_1,\dots,X_n\) is an independent sequence of real random variables, the random vector \((X_1,\dots,X_n)\) satisfies the FKG inequality for the coordinate-wise partial order on \(\R^n\).
    \item If \((X_k)_{k\geq 1}\) is an independent sequence of centred real random variables, \((S_{n})_{n\geq 0}\) is a martingale (for the canonical filtration).
    \item If \(X_1,\dots,X_n\) is an independent sequence of centred real random variables, \(u\in \R\), and \(h:\R^n \to \R\) is a function, \((S_{k})_{k=0}^n\) is a submartingale (for the canonical filtration) under the measures \(P_u(\cdot \given \min_{i=1,\dots,n} S_i-h_i \geq 0)\) and \(P_u(\cdot \given \min_{i=1,\dots,n} S_i-h_i > 0)\). Indeed, denoting \(A_i(r) = \{X_i\geq r\}\), for \(i\geq 0\), one has that, almost surely,
    \begin{multline*}
        E_u\bigl(S_{i+1}\bgiven X_1,\dots,X_i, \cap_{j=1}^n A_j(h_j-S_{j-1})\bigr)
        \\
        =
        S_i + E_u\bigl(X_{i+1}\bgiven S_i, \cap_{j=i+1}^n A_j(h_j-S_{j-1})\bigr)
        \geq
        S_i + E(X_{i+1})
        =
        S_i,
    \end{multline*}
    where we used the Markov's property and that \(\cap_{j=i+1}^n A_j(h_j-S_{j-1})\) is an increasing event for \((X_{i+1},\dots,X_n)\).
\end{enumerate}

Finally, we will often use the following classical large deviation bound.
\begin{lemma}
    \label{lem:very_large_tails_UB}
    Let \(c_0,\delta_0, \sigma_- >0\). Then, there is \(c,\rho\in (0,+\infty)\) such that the following holds. For any \(n\geq 1\), any \(X_1,\dots, X_n\) independent sequence of random variables with
    \begin{equation*}
        E(X_i) = 0,\quad E(X_i^2) \geq \sigma_-^2,\quad \sup_{|z|\leq \delta_0}E(e^{zX_i}) \leq c_0,\quad i=1,\dots, n,
    \end{equation*}
    \begin{equation*}
        P_0\bigl(\bar{S}_{n}\geq t \bigr)
        \leq
        \begin{cases}
            e^{-ct^2/n} & \text{ if } t\leq \rho n,
            \\
            e^{-ct} & \text{ if } t> \rho n.
        \end{cases}
    \end{equation*}
\end{lemma}
\begin{proof}
    First, as the \(X_i\)'s are centred with uniform exponential moments, there is \(r_0>0\) depending only on \(c_0,\delta_0\) such that for \(|z|\leq r_0\),
    \begin{equation*}
        \ln E(e^{z X_i}) = \Var(X_i)\tfrac{z^2}{2} + O(z^3),
    \end{equation*}
    with \(O(z^3)\) being uniform over the law of \(X_i\). Thus, as \(\Var(X_i)\geq \sigma_-^2\), there is \(c>0\) depending only on \(c_0,\delta_0, \sigma_-\) such that for any \(z\) small enough,
    \begin{equation*}
        \ln E(e^{zX_i}) \leq c z^2.
    \end{equation*}
    In particular, for those values of \(z\),
    \begin{equation}
    \label{eq:prf:lem:very_large_tails_UB:MGF}
        E(e^{z S_n})\leq e^{cz^2n}
    \end{equation}
    In particular, taking \(z = \frac{t}{2cn}\) and using Chebychev's inequality, we get that for \(t\) less than \(\rho n\) for some \(\rho>0\),
    \begin{equation*}
        P_0(\bar{S}_{n}\geq t)
        \leq
        \exp(- t^2/4cn),
    \end{equation*}
    which is the first part of the claim. The second part follows again from Chebychev's inequality: taking \(z > 0\) small enough in~\eqref{eq:prf:lem:very_large_tails_UB:MGF}, and using Chebychev's inequality, we get that for \(t\geq \rho n\),
    \begin{equation*}
        P_0(\bar{S}_{n}\geq t) \leq e^{cz^2n}e^{-zt} \leq e^{(cz-\rho )z n}e^{-zt/2} \leq e^{-zt/2}.
        \qedhere
    \end{equation*}
\end{proof}

\section{Trajectory estimates: general walks}
\label{sec:traj_est_walks}

The first class of results concern general walks with free endpoint.

\subsection{Inhomogeneous CLT}
\label{subsec:traj_est_walks:inhomog_CLT}

\begin{theorem}
	\label{thm:inhomog_CLT}
	Let \(A>0\), \(n\geq 1\). Let \(X_1,\dots,X_n\) be an independent sequence of random variables with
	\begin{equation*}
		E(X_i) = 0,\quad E(|X_i|^3)\leq A,\ i=1,\dots, n.
	\end{equation*}
	Then,
	\begin{equation*}
		\sup_{x\in \R} \Big| P\bigl(\tfrac{1}{\sqrt{B_n}}S_{1,n} \leq x\bigr) - P\bigl(\calN(0,1) \leq x\bigr) \Big|
		\leq
		\frac{C A n}{B_n^{3/2}},
	\end{equation*}
	where \(C>0\) is a universal constant.
\end{theorem}
\begin{proof}
	This is a particular case of~\cite[Chapter V, Theorem 3]{Petrov-1995}.
\end{proof}

\subsection{Small ball probabilities}
\label{subsec:traj_est_walks:small_ball_4moments}

The results in the section are all relatively direct consequences of Doob's submartingale inequality and of the inhomogeneous CLT (Theorem~\ref{thm:inhomog_CLT}).

\begin{lemma}
	\label{lem:macro_small_ball_walk}
	Let \(A>0\), \(n\geq 1\) and let \(X_1,\dots, X_n\) be an independent sequence of centred random variables with \(\max_{i=1,\dots,n} E(X_i^2)\leq A\). Then, for any \(\lambda> 0\)
	\begin{equation*}
		P_0\bigl(\max_{i=1,\dots,n} |S_i| \leq \lambda \bigr) \geq 1- \frac{A n}{\lambda^2}.
	\end{equation*}
\end{lemma}
\begin{proof}
	This follows directly from Doob's submartingale inequality.
\end{proof}

\begin{lemma}
	\label{lem:micro_small_ball_walk_LB}
	Let \(A>0\), \(\sigma >0\). There are \(c>0\), \(\lambda_0\geq 0\), \(n_0\geq 1\), such that for any \(n\geq n_0\) and \(X_1,\dots, X_n\) independent sequence of random variables with
	\begin{equation*}
		E(X_i) = 0,
		\quad
		E(|X_i|^3)\leq A,
		\quad
		E(X_i^2) \geq \sigma^2,
	\end{equation*}
	one has the following. For any \(\lambda\geq \lambda_0\),
	\begin{equation*}
		P_0\bigl(\max_{i=1,\dots,n} |S_i| \leq \lambda,\ |S_n|\leq \lambda/2 \bigr) \geq \exp(-\tfrac{c n}{\lambda^2}).
	\end{equation*}
\end{lemma}
\begin{figure}
	\centering
	\includegraphics[width=\textwidth]{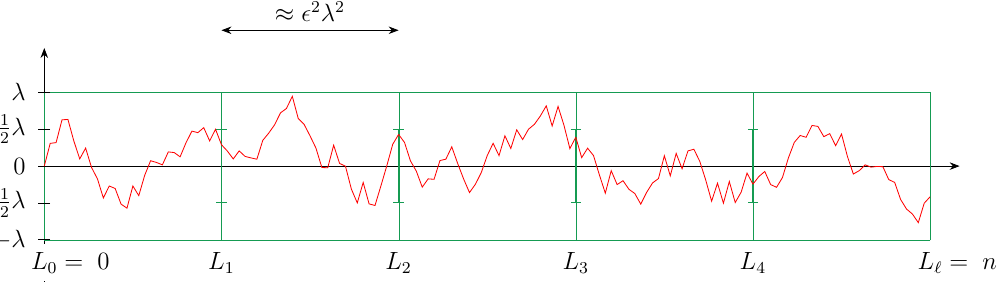}
	\caption{Construction in the proof of Lemma~\ref{lem:micro_small_ball_walk_LB}: the path is forced to pass through the interval \([-\lambda/2,\lambda/2]\) at each time \(L_1,\dots,L_{\ell-1}\).}
	\label{fig:micro_small_ball_walk_LB}
\end{figure}
\begin{proof}
	Let \(\epsilon>0\) to be fixed later. Let \(\ell = \lfloor n (\epsilon\lambda)^{-2} \rfloor\). Let \(0=L_0<L_1<\dots <L_{\ell} = n\) be integers such that
	\begin{equation*}
		(\epsilon\lambda)^2 \geq L_{i}-L_{i-1} \leq 2(\epsilon\lambda)^2,\ i=1,\dots, \ell.
	\end{equation*}
	Note that this requires \(\lambda_0, n_0\) large enough as a function of \(\epsilon\). Now, introduce (see Figure~\ref{fig:micro_small_ball_walk_LB})
	\begin{equation*}
		D_i = \{S_{L_i}\in [-\lambda/2, \lambda/2]\}\cap \{\max_{i=L_{i-1}+1,\dots , L_i} |S_i|\leq \lambda\}.
	\end{equation*}By Lemma~\ref{lem:macro_small_ball_walk}, one has that for any \(i=1,\dots, \ell\), and \(x\in [-\lambda/2, \lambda/2]\),
	\begin{multline*}
		P\bigl(\{\max_{i=L_{i-1}+1,\dots , L_i} |S_i|\leq \lambda\} \bgiven S_{L_{i-1}} = x \bigr)
		\geq
		P\bigl(\{\max_{i=L_{i-1}+1,\dots , L_i} |S_i|\leq \lambda/2\} \bgiven S_{L_{i-1}} = 0 \bigr)
		\\
		\geq
		1-\frac{4A^{2/3} (L_i-L_{i-1})}{\lambda^2}
		\geq
		1- 8 A^{2/3} \epsilon^2.
	\end{multline*}
	Also, by the CLT (Theorem~\ref{thm:inhomog_CLT}), one has
	\begin{equation*}
		P\bigl(S_{L_i}\in [-\lambda/2, \lambda/2] \bgiven S_{L_{i-1}} = x \bigr)
		\geq
		\tfrac{1}{4}
	\end{equation*}
	as soon as \(\epsilon\) is small enough, and \(\lambda_0\) is large enough as a function of \(A,\sigma\). Combining the two, we get
	\begin{equation*}
		P\bigl(D_i \bgiven S_{L_{i-1}} = x \bigr)
		\geq
		\tfrac{1}{8}
	\end{equation*}
	as soon as \(\epsilon\) is small enough, and \(\lambda_0\) is large enough. A direct induction and Markov's property yields
	\begin{equation*}
		P_0\bigl(\max_{i=1,\dots,n} |S_i| \leq \lambda,\ |S_n|\leq \lambda/2 \bigr)
		\geq
		P_0\bigl(\cap_{i=1}^{\ell} D_i\bigr)
		\geq
		8^{-\ell},
	\end{equation*}
	which gives the wanted claim.
\end{proof}

\begin{lemma}
	\label{lem:micro_small_ball_walk_UB}
	Let \(A>0\), \(\sigma >0\). There are \(c>0\), \(\lambda_0\geq 0\), \(n_0\geq 1\), such that for any \(n\geq n_0\) and \(X_1,\dots, X_n\) independent sequence of random variables with
	\begin{equation*}
		E(X_i) = 0,
		\quad
		E(|X_i|^3)\leq A,
		\quad
		E(X_i^2) \geq \sigma^2,
	\end{equation*}
	one has the following. For any \(\sqrt{n} \geq \lambda\geq \lambda_0\),
	\begin{equation*}
		P_0\bigl(\max_{i=1,\dots,n} |S_i| \leq \lambda \bigr) \leq \exp(-\tfrac{c n}{\lambda^2}).
	\end{equation*}
\end{lemma}
\begin{figure}
	\centering
	\includegraphics{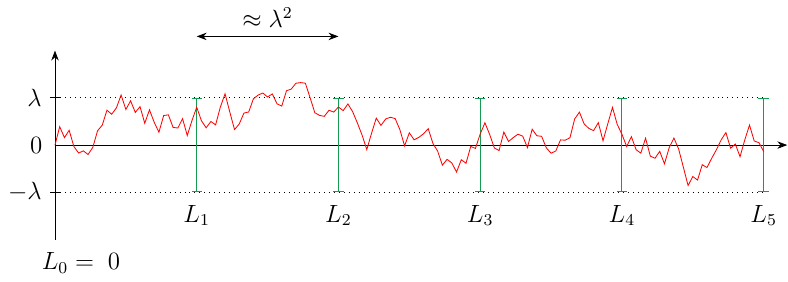}
	\caption{Construction in the proof of Lemma~\ref{lem:micro_small_ball_walk_UB}: the path is forced to remain in the interval \([-\lambda,\lambda]\) and has to pass through the interval \([-\lambda,\lambda]\) at each time \(L_1,\dots,L_{\ell}\), but is unconstrained otherwise.}
	\label{fig:micro_small_ball_walk_UB}
\end{figure}
\begin{proof}
	Let \(\ell = \lceil n \lambda^{-2} \rceil\). Let \(0=L_0<L_1<\dots <L_{\ell} = n\) be integers such that
	\begin{equation*}
		\tfrac{1}{2}\lambda^2 \leq L_{i}-L_{i-1} \leq \lambda^2,\ i=1,\dots, \ell.
	\end{equation*}
	Now, for any \(|x|\leq \lambda\), and any \(i\in \{1,\dots,\ell\}\),
	\begin{equation*}
		P\bigl(|S_{L_i}| \leq \lambda \bgiven S_{L_{i-1}} = x\bigr)
		\leq
		e^{-c}
	\end{equation*}
	for some \(c>0\) depending on \(A,\sigma\) by the CLT (Theorem~\ref{thm:inhomog_CLT}). Then, a direct induction yields
	\begin{equation*}
		P_0\bigl(\max_{i=1,\dots,n} |S_i| \leq \lambda \bigr)
		\leq
		P_0\bigl(\cap_{i=1}^{\ell} \{|S_{L_i}| \leq \lambda\} \bigr)
		\leq
		e^{-c\ell},
	\end{equation*}
	which is the claim.
\end{proof}

\subsection{Positivity probabilities}
\label{subsec:traj_est_walks:positivity_4moments}

This section builds on the proof presented in Appendix~\ref{app:bounded_increm_pos_proba}. We therefore encourage the reader to first look at Appendix~\ref{app:bounded_increm_pos_proba} before reading the present section.

\subsubsection*{Lower bound}

We start by a preliminary ``truncation Lemma''.
\begin{lemma}
	\label{lem:truncated_RV}
	Let \(\alpha>1\). Let \(A>0\). Let \(X\) be a random variable such that \(E(X) = 0\) and \(E(|X|^{\alpha})\leq A\). Then, for any \(K\geq 1\), one can construct a random variable \(Y\) and a coupling \(Q\) of \(X,Y\) such that
	\begin{equation*}
		E(Y) = 0,
		\quad
		P(|Y|\leq (A+1)K) = 1,
		\quad
		Q(X\neq Y)\leq \frac{A+1}{K^{\alpha}}.
	\end{equation*}
	Moreover, for \(p\geq 1\),
	\begin{gather*}
		E(Y^p) \leq 2^{p-1} \bigl(E(|X|^p) + K^{p-\alpha}A^p\bigr),
		\\
		E(Y^2) \geq E(X^2) - AK^{2-\alpha} - 2A^2 K^{2-2\alpha} \qquad \text{ if } \alpha > 2.
	\end{gather*}
\end{lemma}
\begin{proof}
	Let \(X\) be as in the statement. Let \(K>0\). Let then \(\xi\) be a Bernoulli random variable of parameter \(K^{-\alpha}\) defined on the same space as \(X\) and independent from \(X\). Set \(x = K^{\alpha}E(X\mathds{1}_{|X|> K})\) and define
	\begin{equation*}
		Y = X \mathds{1}_{|X|\leq K} + x \xi.
	\end{equation*}
	It is centred:
	\begin{equation*}
		E(Y)
		=
		E(X\mathds{1}_{|X|\leq K}) + x K^{-\alpha}
		=
		-E(X\mathds{1}_{|X|> K}) + x K^{-\alpha}
		=
		0.
	\end{equation*}
	The variable \(X\mathds{1}_{|X|\leq K}\) is taking values in \([-K,K]\), so we need to check that \(|x|\) is not too large. By Hölder's and Chebychev's inequalities, \(x\) satisfies
	\begin{equation*}
		|x|
		\leq
		K^{\alpha}E(|X|^{\alpha})^{1/\alpha}P(|X|> K)^{(\alpha-1)/\alpha}
		\leq
		K^{\alpha}E(|X|^{\alpha})^{1/\alpha}(E(|X|^{\alpha}) K^{-\alpha} )^{(\alpha-1)/\alpha}
		\leq
		KA,
	\end{equation*}
	which is the second property of the wanted random variable. Then,
	\begin{equation*}
		P(X\neq Y)
		\leq
		P\bigl(\{\xi = 1\}\cup\{|X|>K\}\bigr)
		\leq
		K^{-\alpha} + P(|X|>K)
		\leq
		K^{-\alpha} \bigl( 1+ A \bigr),
	\end{equation*}
	which is the third property. Now, for \(p >1\), Jensen inequality gives
	\begin{equation*}
		E(|Y|^p)
		\leq
		2^{p-1}\bigl(E(|X|^p\mathds{1}_{|X|\leq K}) + |x|^pE(\xi)\bigr)
		\leq
		2^{p-1} \bigl(E(|X|^p) + K^{p-\alpha}A^p\bigr).
	\end{equation*}
	Finally, for \(\alpha>2\),
	\begin{multline*}
		E(Y^2)
		=
		E(X^2) -E(X^2\mathds{1}_{|X|> K}) + x^2K^{-\alpha} - 2x^2 K^{-2\alpha}
		\\
		\geq
		E(X^2) - AK^{2-\alpha} - 2A^2 K^{2-2\alpha}.
		\qedhere
	\end{multline*}
\end{proof}

\begin{lemma}
	\label{lem:positivity_general_walks:LB}
	Let \(A>0, A'>0, s>0\). Then, there is \(c>0\) such that for any \(n\geq 1\), \(X_1,\dots, X_n\) independent sequence of real random variables with
	\begin{equation*}
		E(X_i) = 0,
		\quad
		E(X_i^4)\leq A,
		\quad
		E(X_i\mathds{1}_{X_i>0})\geq s,
		\ i=1,\dots, n,
	\end{equation*}
	and any \(0\leq u\leq A'\sqrt{n}\),
	\begin{equation*}
		P_u( \tau >n )
		\geq
		\frac{c(u+1)}{\sqrt{n}}.
	\end{equation*}
\end{lemma}
\begin{proof}
	We follow the same strategy as the lower bound for bounded random variables of Appendix~\ref{app:bounded_increm_pos_proba}. There, the key was that upward increments were bounded, but a direct inspection of the argument shows that the real ingredient was that for \(n\) large enough, upward increments were bounded by \(\sqrt{n}\). For \(i=1,\dots, n\), let \(Y_i\) be the random variable obtained via Lemma~\ref{lem:truncated_RV} with \(K= \sqrt{n}\), \(\alpha = 4\). As the \(X_i\)'s form an independent family, we can assume that the \(Y_i's\) are defined on the same space as the \(X_i's\), and that \(\bigl((X_i,Y_i)\bigr)_{i=1}^{n}\) forms an independent family of random vectors.
	
	Let \(Z_0 = S_0\) and \(Z_{i} = Z_{i-1}+ Y_i\) be the random walk associated to the \(Y_i\)'s. Define
	\begin{equation*}
		\tau' = \min\{k\geq 1:\ Z_k<0\}.
	\end{equation*}
	Let \(D = \cap_{i=1}^{n}\{X_i = Y_i\}\).
	We then have
	\begin{equation*}
		P_u(\tau > n)
		\geq
		P_u(\tau'> n,\ D)
		\geq
		P_u(\tau' > n) - P(D^c).
	\end{equation*}
	Now, \(P(D^c)\leq \frac{(A+1)n}{n^{2}}\) by Lemma~\ref{lem:truncated_RV} and a union bound. We therefore need to lower bound \(P_u(\tau' > n)\). Note that as the \(Y_i\)'s are centred, \(Z\) is a martingale. Thus, by the optional stopping Theorem, for \(u>0\),
	\begin{multline*}
		u
		=
		E_u(Z_{\tau'\wedge n} )
		=
		P_u(\tau' >n)E_u(Z_{n} \given \tau' >n) + E_u(Z_{\tau'}\mathds{1}_{\tau'\leq n})
		\\
		\leq
		P_u(\tau' >n)E_u(Z_{n} \given \tau' >n).
	\end{multline*}
	Now, let \(L\) be some large number to be fixed later and define \(T = \min\{k\geq 0:\ Z_k \geq L\sqrt{n} \}\). One has
	\begin{multline*}
		E_u(Z_{n} \given \tau' >n)
		=
		\mathds{1}_{u\geq L\sqrt{n}}\bigl( u + E_u(Z_{1,n} \given \tau' >n)\bigr)
		+
		\mathds{1}_{u< L\sqrt{n}}\sum_{k=1}^n E_u(\mathds{1}_{T=k} Z_n\given \tau' > n) \\
		+ \mathds{1}_{u< L\sqrt{n}} E_u(\mathds{1}_{T>n} Z_n\given \tau' > n).
	\end{multline*}
	Since \(\mathds{1}_{T>n} Z_n < L\sqrt{n}\), the third term is smaller than \(L\sqrt{n}\).
	Moreover, using Markov's property, we obtain
	\begin{align*}
		E_u(\mathds{1}_{T=k} Z_n\given \tau' > n)
		&=
		E_u\bigl( \mathds{1}_{T=k} \bigl(Z_{k-1} + Y_k + E(Z_{k+1,n} \given Z_k, \tau' >n)\bigr) \bgiven \tau'>n \bigr)
		\\
		&\leq
		E_u\bigl( \mathds{1}_{T=k} \bigl((A+1+L)\sqrt{n} + E(Z_{k+1,n} \given Z_k, \tau' >n)\bigr) \bgiven \tau'>n \bigr).
	\end{align*}
	Now, for any \(v\geq L\sqrt{n}\), \(k\geq 0\),
	\begin{equation*}
		E(Z_{k+1,n} \given Z_k=v, \tau' >n)
		\leq
		\frac{E( Z_{k+1,n}^2)^{1/2}}{P( \min_{i=k+1,\dots, n} Z_i \geq 0 \given Z_k=v)}.
	\end{equation*}
	Moreover, for any \(v\geq L\sqrt{n}\),
	\begin{equation*}
		P( \min_{i=k+1,\dots, n} Z_i \geq 0 \given Z_k=v)
		\geq
		P(\max_{i=k+1,\dots, n} |Z_{k+1,i}| \leq L\sqrt{n})
		\geq
		1- \frac{E(Z_{k+1,n}^2)}{L^2 n}
	\end{equation*}
	by Doob's submartingale inequality. As \(E(Z_{k+1,n}^2) \leq 3n \sqrt{A}\) for \(n\) large enough by Lemma~\ref{lem:truncated_RV}, we obtain that for the choice \(L = 2A^{1/4}\), and \(v\) as before,
	\begin{equation*}
		E(Z_{k+1,n} \given Z_k=v, \tau' >n) \leq 3^{3/2} A^{1/4} \sqrt{n}.
	\end{equation*}
	Using this, we obtain
	\begin{equation*}
		E_u(\mathds{1}_{T=k} Z_n\given \tau' > n)
		\leq
		\sqrt{n}(A + 1 + 2A^{1/4} + 3^{3/2} A^{1/4})P_u(T=k\given \tau' > n),
	\end{equation*}
	and thus
	\begin{equation*}
		E_u(Z_{n} \given \tau' >n)
		\leq
		\mathds{1}_{u\geq L\sqrt{n}}(u + C\sqrt{n})
		+
		\mathds{1}_{u< L\sqrt{n}} C\sqrt{n},
	\end{equation*}
	for some \(C>0\) depending only on \(A\). Using this in the equation obtained via the optional stopping Theorem, we get that for \(u\leq A'\sqrt{n}\),
	\begin{equation*}
		P_u(\tau' >n) \geq \frac{u}{(C+A')\sqrt{n}},
	\end{equation*}
	therefore, for \(n\) large enough as a function of \(A\),
	\begin{equation*}
		P_u(\tau > n)
		\geq
		\frac{u}{(C+A')\sqrt{n}} - \frac{A+1}{n} \geq \frac{u}{2(C+A')\sqrt{n}},
	\end{equation*}
	for any \(A'\sqrt{n}\geq u\geq 1\). The cases \(u\in [0,1)\), and \(n\) small are handled as in proof of Theorem~\ref{thm:bounded_increm_pos_proba} using the lower bound on \(E(X_i \mathrm{1}_{X_i>0})\).
\end{proof}

\subsubsection*{Moments bounds}

\begin{lemma}
	\label{lem:positivity_general_walks:moments}
	Let \(A>0, s>0\). Then, there are \(c,c'>0\) such that for any \(n\geq 1\), \(X_1,\dots, X_n\) independent sequence of real random variables with
	\begin{equation*}
		E(X_i) = 0,
		\quad
		E(X_i^4)\leq A,
		\quad
		E(X_i\mathds{1}_{X_i>0})\geq s,
		\ i=1,\dots, n,
	\end{equation*}
	one has for any \(u\geq 0\), \(k=1,\dots,n\),
	\begin{gather*}
		E_u\bigl(\max_{i=1,\dots, n}S_{i}^2 \bgiven \tau >n \bigr) \leq 12u^2 + c n,
		\\
		E_u\bigl( S_{k} \bgiven \tau >n \bigr) \geq c' (\sqrt{k}+u).
	\end{gather*}
\end{lemma}
\begin{figure}
	\centering
	\includegraphics{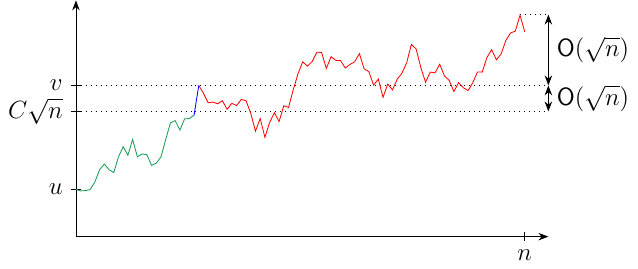}
	\caption{Upper bound in Lemma~\ref{lem:positivity_general_walks:moments}: If the random walk trajectory reaches a height above \(C\sqrt{n}\), then splitting the path at the first such point, as in the picture, shows that the final height is the sum of three terms of order at most \(\sqrt{n}\).}
\end{figure}
\begin{proof}
	Start with the lower bound. As \(\{\tau >n\}\) is an increasing event in the \(X_i\)'s, and \(\max(S_{k}, 0)\) is an increasing functions of the \(X_i\)'s, we have that by the FKG inequality,
	\begin{multline*}
		E_u\bigl( S_{k} \given \tau >n \bigr)
		=
		E_u\bigl( \max(S_{k},0) \given \tau >n \bigr)
		\geq
		E_u\bigl( \max(S_{k},0) \bigr)
		\\
		=
		E_0\bigl( \max(S_{k}+u,0) \bigr)
		\geq
		P_0(S_k\geq \sqrt{k})(\sqrt{k} + u).
	\end{multline*}But \(P_0(S_k\geq \sqrt{k})>c'\) for some \(c'>0\) depending only on \(A,s\) by the inhomogeneous CLT (Theorem~\ref{thm:inhomog_CLT}).
	
	Turn now to the upper bound. Set \(A'=\sqrt{A}\). Under \(P_u(\cdot \given \tau >n )\), \((S_i)_{i= 1}^n\) is a submartingale (see Section~\ref{subsection:mart_submart_FKG}). Thus, by Doob's submartingale inequality,
	\begin{equation*}
		E_u\bigl(\max_{i=1,\dots, n}S_i^{2} \given \tau >n\bigr)
		\leq
		4E_u\bigl(S_n^2 \bgiven \tau >n \bigr).
	\end{equation*}
	Now, define
	\begin{equation*}
		\tau' = \min\{k\geq 0:\ S_k \geq \sqrt{2A'n}\}.
	\end{equation*}
	Then,
	\begin{equation*}
		E_u\bigl(S_n^2 \bgiven \tau >n\bigr)
		\leq
		2A' n + \sum_{k=0}^n E_u\bigl(S_n^2 \mathds{1}_{\tau'=k}\bgiven \tau >n\bigr).
	\end{equation*}
	Now, for any \(v> \sqrt{2A'n}\) and \(k\in \{0,1,\dots, n-1\}\),
	\begin{multline*}
		E\bigl(S_{k+1,n}^2 \bgiven \min_{i=k,\dots, n} S_i >0,\ S_k = v\bigr)
		\leq
		\frac{E(S_{k+1,n}^2)}{P\bigl(\min_{i=k,\dots, n} S_i >0 \bgiven S_k = v\bigr)}
		\\
		\leq
		\frac{A'(n-k)}{P\bigl(\max_{i=k,\dots, n} |S_{k,i}| \leq \sqrt{2A'n}\bigr)}
		\leq
		\frac{A'(n-k)}{1-\tfrac{A'(n-k)}{2A' n}}
		\leq
		2A'n,
	\end{multline*}
	by our Hypotheses and Lemma~\ref{lem:macro_small_ball_walk}. Then, as \(S_n^2 = (S_{k-1}+ X_k +S_{k+1,n})^2 \leq 3(S_{k-1}^2 + X_k^2 +S_{k+1,n}^2)\) for \(k\geq 1\), one has that by the previous display, setting \(X_0=0\),
	\begin{equation*}
		E_u\bigl(S_n^2 \mathds{1}_{\tau'=k}\bgiven \tau >n\bigr)
		\leq
		3E_u\bigl((u^2+2A'n + X_k^2 +2A'n) \mathds{1}_{\tau'=k}\bgiven \tau >n\bigr).
	\end{equation*}
	Now,
	\begin{align*}
		\sum_{k=1}^n E_u\bigl( X_k^2 \mathds{1}_{\tau'=k}\bgiven \tau >n\bigr)
		&\leq
		\sum_{k=1}^n E_u\bigl( X_k^4 \bgiven \tau >n\bigr)^{\tfrac{1}{2}}P_u\bigl(\tau'=k\bgiven \tau >n\bigr)^{\tfrac{1}{2}}
		\\
		&\leq
		cA'n^{1/2}\sum_{k=1}^n P_u\bigl(\tau'=k\bgiven \tau >n\bigr)^{\tfrac{1}{2}}
		\leq
		cA'n
	\end{align*}
	where we used Cauchy-Schwartz inequality in the first inequality, Jensen's in the third, and the second follows from Lemma~\ref{lem:positivity_general_walks:LB} and our Hypotheses:
	\begin{equation*}
		E_u\bigl( X_k^4 \bgiven \tau >n\bigr)
		\leq
		\frac{E( X_k^4 )}{P_u(\tau >n)}
		\leq
		cA\sqrt{n}.
	\end{equation*}
	Combining everything, we obtained
	\begin{equation*}
		E_u\bigl(S_n^2 \bgiven \tau >n\bigr)
		\leq
		3u^2 + 14A'n + 3\sum_{k=1}^n E\bigl( X_k^2 \mathds{1}_{\tau'=k}\bgiven \tau >n\bigr)
		\leq
		3u^2 + C n
	\end{equation*}
	for some \(C\in (0,+\infty)\). This is the claim.
\end{proof}

\subsubsection*{Upper bound}

\begin{lemma}
	\label{lem:positivity_general_walks:UB}
	Let \(A>0, s>0\). Then, there is \(c>0\) such that for any \(n\geq 1\), \(X_1,\dots, X_n\) independent sequence of real random variables with
	\begin{equation*}
		E(X_i) = 0,
		\quad
		E(X_i^4)\leq A,
		\quad
		E(X_i\mathds{1}_{X_i>0})\geq s,
		\ i=1,\dots, n,
	\end{equation*}
	one has for any \(u\geq 0\),
	\begin{equation*}
		P_u\bigl( \tau >n \bigr) \leq \frac{c(u+1)}{\sqrt{n}}.
	\end{equation*}
\end{lemma}
\begin{proof}
	We proceed in several steps. We first prove an upper bound valid for large values of the starting point.
	\begin{claim}
		\label{claim:prf_lem:positivity_proba_UB:poly}
		There is \(C>0\) depending only on \(A,s\) such that for \(u\geq n^{1/4}\),
		\begin{equation*}
			P_u(\tau > n) \leq C\frac{ u}{\sqrt{n}}.
		\end{equation*}
	\end{claim}
	\begin{proof}
		First, by the optional stopping Theorem, for \(u>0\),
		\begin{multline*}
			u
			=
			E_u(S_{\tau'\wedge n})
			=
			P_u(\tau >n)E(S_{n} \given \tau >n) + \sum_{k=1}^{n}E_u(S_{k} \mathds{1}_{\tau=k} )
			\\
			\geq
			P_u(\tau >n)c'\sqrt{n} - \sum_{k=1}^{n}A^{1/4} P(\tau=k)^{3/4}
			\geq
			P_u(\tau >n)c'\sqrt{n} - A^{1/4}n^{1/4},
		\end{multline*}
		where \(c'>0\) depends only on \(A,s\), and we used Lemma~\ref{lem:positivity_general_walks:moments}, \(S_{\tau} \geq -|X_{\tau}|\), Hölder's inequality, and our moment assumption in the first inequality, and Jensen's inequality in the second.
		Re-arranging, this implies
		\begin{equation*}
			P_u(\tau >n)
			\leq
			\frac{u + A^{1/4}n^{1/4} }{c'\sqrt{n}},
		\end{equation*}
		which is the wanted claim.
	\end{proof}
	Then, we upgrade this to a bound on the tails of \(\tau\) up to a log-correction.
	\begin{claim}
		\label{claim:prf_lem:positivity_proba_UB:log}
		There are \(C,c>0\) depending only on \(A,s\) such that for \(u\geq 0\),
		\begin{equation*}
			P_u(\tau > n) \leq C(u+1)\frac{(\ln(n+1))^c}{\sqrt{n}}.
		\end{equation*}
	\end{claim}
	\begin{proof}
		Suppose \(u< n^{1/4}\) (otherwise the claim follows from Claim~\ref{claim:prf_lem:positivity_proba_UB:poly}). Let \(r=1/2\). Let \(L= \lceil n^{r}\rceil\). First, by Claim~\ref{claim:prf_lem:positivity_proba_UB:poly}, for \(n\) larger than some universal \(n_0\geq 3\) (\(3\) is taken for later convenience and bears no particular importance),
		\begin{align*}
			P_u(\tau >n)
			&=
			E_{u}\bigl( \mathds{1}_{\tau >L} P( \tau >n \given S_1,\dots, S_L)\bigr)
			\\
			&\leq
			E_{u}\bigl( \mathds{1}_{\tau >L} (\mathds{1}_{S_L\geq n^{r/2}}P( \min_{j=L+1,\dots,n}S_{j}\geq 0 \given S_L) \\
			&\qquad +  \mathds{1}_{S_L\leq n^{r/2}}P( \min_{j=L+1,\dots,n}S_{j}\geq 0 \given S_L= n^{r/2}) )\bigr)
			\\
			&\leq
			E_{u}\bigl( \mathds{1}_{\tau >L} (\mathds{1}_{S_L\geq n^{r/2}} \tfrac{CS_L}{\sqrt{n-L}} +  \mathds{1}_{S_L\leq n^{r/2}}\tfrac{Cn^{r/2}}{\sqrt{n-L}} )\bigr)
			\\
			&\leq
			\tfrac{2C}{\sqrt{n}} E_{u}\bigl( \mathds{1}_{\tau >L}(S_L + n^{r/2})\bigr)
		\end{align*}
		as \(v\mapsto P( \min_{j=L+1,\dots,n}S_{j}\geq 0 \given S_L= v)\) is increasing in \(v\). But now, by Lemma~\ref{lem:positivity_general_walks:moments},
		\begin{equation*}
			E_{u}\bigl( S_L\bgiven \tau >L\bigr)
			\leq
			u + c\sqrt{L}
			\leq
			c'n^{r/2}.
		\end{equation*}
		So, we obtained that there is \(c\geq 1\) such that for \(u\leq n^{r/2}\), \(n\geq n_0\),
		\begin{equation*}
			P_u(\tau >n)
			\leq
			\tfrac{cn^{r/2}}{\sqrt{n}}P_u(\tau > n^{r}).
		\end{equation*}
		Define
		\begin{equation*}
			k_u = \min\big\{k\geq 1 :\ n^{r^k} < \max(u^{4},n_0) \big\}.
		\end{equation*}
		By the previously obtained bound and a direct induction, we obtain
		\begin{equation*}
			P_u(\tau >n)
			\leq
			\tfrac{cn^{r/2}}{\sqrt{n}}P_u(\tau > n^{r})
			\leq
			\tfrac{cn^{r/2}}{\sqrt{n}}\cdot \tfrac{cn^{r^2/2}}{n^{r/2}} P_u(\tau > n^{r^2})
			\leq
			\dots
			\leq
			\tfrac{c^{k_u}}{\sqrt{n}} n^{r^{k_u}/2}P_u(\tau > n^{r^{k_u}}).
		\end{equation*}
		Now, \(k_u\leq \frac{\ln(\ln(n))}{|\ln(r)|}\), so for \(u\leq n_0^{r/2}\),
		\begin{equation*}
			P_u(\tau >n)
			\leq
			C\frac{(\ln(n))^c}{\sqrt{n}}
		\end{equation*}
		for some \(C,c>0\), as \(n^{r^{k_u}}\leq n_0\) in that case. On the other hand, if \(u> n_0^{r/2}\),
		\begin{equation*}
			P_u(\tau >n)
			\leq
			\tfrac{c^{k_u}}{\sqrt{n}} n^{r^{k_u}/2}P_u(\tau > n^{r^{k_u}})
			\leq
			C\tfrac{c^{k_u}}{\sqrt{n}} u
			\leq
			Cu\tfrac{(\ln(n))^c}{\sqrt{n}}
		\end{equation*}
		for some \(C,c>0\), by Claim~\ref{claim:prf_lem:positivity_proba_UB:poly} as \(n^{r^{k_u}} \leq u^{4}\) in this case.
	\end{proof}
	\begin{claim}
		\label{claim:prf_lem:positivity_proba_pointwise_UB}
		Let \(\epsilon>0\). There is \(c>0\) depending only on \(A,s,\epsilon\) such that for \(u\geq 0\),
		\begin{equation*}
			P_u(\tau = n) \leq \frac{c(u+1)}{n^{\frac{3}{2} -\epsilon}}.
		\end{equation*}
	\end{claim}
	\begin{proof}
		Let \(f(k) = C\frac{(\ln(k+1))^c}{\sqrt{k}}\) with \(C,c\) given by Claim~\ref{claim:prf_lem:positivity_proba_UB:log}. We can always suppose \(n\) large enough. First, we claim that there is \(c>0\) depending only on \(A,s\) such that for any \(n\),
		\begin{equation}
			\label{eq:prf:claim:prf_lem:positivity_proba_pointwise_UB:log_positiv_endpoint}
			\sup_{x\in \R}P_u\bigl(\tau >n,\ S_n\in [x,x+1]\bigr) \leq \frac{c(u+1)}{\sqrt{n}}f(n).
		\end{equation}
		Indeed, letting \(L=\lfloor n/2\rfloor\), we can use inclusion of events and the Markov property to write
		\begin{equation*}
			P\bigl(\tau >n,\ S_n\in [x,x+1]\bigr)
			\leq
			\int_{0}^{\infty} dy P_u\bigl(\tau >L,\ S_L\in dy\bigr) P(S_{n}\in [x,x+1]\given S_L=y).
		\end{equation*}
		But, using the triangle inequality and Theorem~\ref{thm:inhomog_CLT}, \(P(S_{n}\in [x,x+1]\given S_L=y)\leq \frac{c}{\sqrt{n}}\) uniformly over \(x,y\) (a term \(c/\sqrt{n}\) comes from the error committed by replacing \(S_n-S_L\) by a centred Gaussian with the same variance, and another \(c/\sqrt{n}\) comes from upper bounding the event that the Gaussian is in \([x-y,x-y+1]\)). So,
		\begin{multline*}
			\int_{0}^{\infty} dy P_u\bigl(\tau >L,\ S_L\in dy\bigr) P(S_{n}\in [x,x+1]\given S_L=y)
			\\
			\leq
			\frac{c}{\sqrt{n}}P_u(\tau >L)
			\leq
			\frac{c}{\sqrt{n}}(u+1)f(L),
		\end{multline*}
		by Claim~\ref{claim:prf_lem:positivity_proba_UB:log}. Introduce the backward walk
		\begin{equation*}
			\backvec{S}_{k} = \backvec{S}_{k-1} - X_{n-k}.
		\end{equation*}
		Still letting \(L=\lfloor n/2\rfloor\), \(L' = n-L-1\), we therefore have that by Markov property and our Hypotheses on the moments of the \(X_i\)'s,
		\begin{align*}
			&P_u(\tau = n)
			=
			E_u\Bigl( \mathds{1}_{\tau>L} E\Bigl(\mathds{1}_{\min_{i=L+1,\dots, n-1} S_{i} \geq 0} P(X_n< -S_{n-1}\given S_{n-1}) \Bgiven S_L \Bigr) \Bigr)
			\\
			&\quad\leq
			E_u\Bigl( \mathds{1}_{\tau>L} E\Bigl(\mathds{1}_{\min_{i=L+1,\dots, n-1} S_{i} \geq 0} \frac{c}{(1+S_{n-1})^4} \Bgiven S_L \Bigr) \Bigr)
			\\
			&\quad=
			\sum_{k=0}^{\infty} E_u\Bigl( \mathds{1}_{\tau>L} \mathds{1}_{S_{n-1}\in [k,k+1)} \frac{c}{(1+S_{n-1})^4} P\Bigl(\min_{i=L+1,\dots, n-1} S_{i} \geq 0 \Bgiven S_L, S_{n-1} \Bigr) \Bigr)
			\\
			&\quad\leq
			\sum_{k=0}^{\infty} \frac{c}{(1+k)^4} E_u\Bigl( \mathds{1}_{\tau>L} \mathds{1}_{S_{n-1}\in [k,k+1)} P\Bigl(\min_{i=1,\dots,L'} \backvec{S}_{i} \geq -S_{n-1}   \Bgiven \backvec{S}_{L'} = S_L-S_{n-1}, \backvec{S}_{0}= 0 \Bigr) \Bigr)
			\\
			&\quad\leq
			\sum_{k=0}^{\infty} \frac{c}{(1+k)^4} E_u\Bigl( \mathds{1}_{\tau>L} P\Bigl(\backvec{S}_{L'} \in (S_L,S_L+1],\ \min_{i=1,\dots,L'} \backvec{S}_{i} \geq 0   \Bgiven \backvec{S}_{0}= k+1 \Bigr) \Bigr)
			\\
			&\quad\leq
			\sum_{k=0}^{\infty} \frac{c}{(1+k)^4} \frac{c(k+2)}{\sqrt{L'}} f(L') P_u(\tau>L)
			\leq
			C\frac{(u+1)}{\sqrt{n}}f(n)^2,
		\end{align*}
		where we globally translated \(\backvec{S}\) by \(k+1\) in the third inequality, and used~\eqref{eq:prf:claim:prf_lem:positivity_proba_pointwise_UB:log_positiv_endpoint}, and Claim~\ref{claim:prf_lem:positivity_proba_UB:log} in the last line. Using \((\ln(n+1))^a \leq c(\epsilon,a) n^{\epsilon}\) for any \(a,\epsilon>0\) gives the wanted claim.
	\end{proof}
	We are now ready to conclude. By Doob's optional stopping Theorem,
	\begin{equation*}
		u
		=
		E_u(S_{\tau \wedge n})
		=
		P_u(\tau>n)E_u(S_n\given \tau >n) + \sum_{k=1}^{n} E_{u}(S_k\mathds{1}_{\tau=k}).
	\end{equation*}
	But now,
	\begin{multline*}
		\sum_{k=1}^{n} E_{u}(S_k\mathds{1}_{\tau=k})
		\geq
		- \sum_{k=1}^{n} E_{u}(|X_k|\mathds{1}_{\tau=k})
		\geq
		- \sum_{k=1}^{n} E(|X_k|^4)^{1/4}P_u(\tau=k)^{3/4}
		\\
		\geq
		- A^{1/4} c\sum_{k=1}^{n} \tfrac{(u+1)^{3/4}}{k^{33/32}}
		\geq
		- C (u+1)^{3/4},
	\end{multline*}
	where we used Claim~\ref{claim:prf_lem:positivity_proba_pointwise_UB} with \(\epsilon=1/8\) in the third inequality. Combining this with the fact that by Lemma~\ref{lem:positivity_general_walks:moments}, \(E_u(S_n\given \tau >n)\geq c'\sqrt{n}\) for some \(c'>0\) depending only on \(A,s\), we obtain
	\begin{equation*}
		P_u(\tau>n)
		\leq
		\frac{u + C (u+1)^{3/4}}{c' \sqrt{n}}
		\leq
		\frac{c(u+1)}{\sqrt{n}}.
		\qedhere
	\end{equation*}
\end{proof}

\section{Inhomogeneous Local Limit Theorem}
\label{sec:LLT}

The main tool we will use to convert ``free endpoint'' estimates to estimates on bridges or excursions is an inhomogeneous version of the Local Limit Theorem.

\begin{theorem}
	\label{thm:inhomo_LLT}
	Let \(\delta_0 >0\), \(c_0<\infty\), \(a\in [0,1]^{\Z}\) be irreducible and aperiodic. Let \(\alpha\in [0,2/3)\). Then, there exist \(n_0 \geq 1\), \(C\geq 0\), such that for any \(X_1,X_2,\dots\) independent sequence of \(\Z\)-valued random variables with laws in \(\calM_{\delta_0,c_0}^a\), the following holds. For any \(n\geq n_0\), \(y\in \Z\) with \(\abs{y-E_0(S_n)}\leq n^{\alpha}\)
	\begin{equation*}
        \exp(-Cn^{-\min(2-3\alpha,1/3)})
		\leq
		\sqrt{2\pi B_{n}}e^{(y-m_{n})^2/2B_{n}} P_0(S_{n} = y)
		\leq
		\exp(Cn^{-\min(2-3\alpha,1/3)})
	\end{equation*}
	where \(B_{k} = \sum_{i=1}^{k} \Var(X_i)\), \(m_{k} = E(S_k)\).
\end{theorem}
\begin{proof}
	Use the shorthands \(\calM \equiv \calM_{\delta_0,c_0}^a\) and \(P\equiv P_0\).
	We start by some observations/definitions which follow from our Hypotheses, and that will be used in the proof.
	\begin{enumerate}[series=LLTPrfObs]
		\item \label{LLT:prf:unif_variances} For every \(i\geq 1\),
		\begin{equation*}
			0< \sigma_-^2 \coloneqq \inf_{p\in \calM} \sigma_p^2 \leq \sup_{p\in \calM} \sigma_p^2 \eqqcolon \sigma_+^2 <\infty.
		\end{equation*}
		\item \label{LLT:prf:analyticity_radius} There exists \(\delta_0 \geq 4r>0\) such that for all \(p\in \calM\), \(z\mapsto H_p(z)\) is analytic in the disc \(\bbD_{4r} = \setof{z\in \bbC}{\abs{z}<4r}\). Fix such an \(r\) for the rest of the proof.
		\item \label{LLT:prf:norm_charac_funct} For any \(b>0\), there exists \(c_b>0\) such that
		\begin{equation*}
			\sup_{p\in \calM}\sup_{\lambda\in [-r,r]}\sup_{b\leq \abs{\theta} \leq \pi} \Babs{\frac{M_p(\lambda +\rmi \theta)}{M_p(\lambda)}} \leq e^{-c_b}.
		\end{equation*}
		Indeed, without the sup over \(p\), it is a consequence of aperiodicity of \(p\). The bound with the sup over \(p\) then follows from compactness of \(\calM\).
	\end{enumerate}
	Introduce then
	\begin{equation*}
		c_r = \frac{1}{4r^3} \sup_{p\in \calM} \sup_{\abs{z}\leq 3r} |H_p(z)|,
	\end{equation*}
	with \(r\) given in Observation~\ref{LLT:prf:analyticity_radius}. Note that the sup above is finite by compactness of \(\calM\) (recall that \(H_p\) is well defined and analytic on \(\bbD_{4r}\) for any \(p\in \calM\)). This will be used to bound the error term in second-order Taylor approximations of \(H_p\)'s. Note already that for any \(\abs{z}\leq r\), Cauchy integral formula gives
	\begin{equation}\label{eq:boundH3}
		\abs{H_p^{(3)}(z)}
        \leq
        \frac{6}{(2r)^3}\sup_{|w-z|= 2r} |H_p(w)|
        \leq
        3c_r,
	\end{equation}
	as \(\setof{w}{\abs{z-w} \leq 2r}\subset \bbD_{3r}\) for \(\abs{z}\leq r\).
	\begin{enumerate}[resume=LLTPrfObs]
		\item \label{LLT:prf:expansion_H_p} One has that for any \(p\in \calM\), \(\lambda\in (-r,r)\), \(z\in \bbD_r\),
		\begin{equation*}
			\Babs{H_p(\lambda + z)- H_p(\lambda) - z H_p'(\lambda) - \frac{z^2}{2}H_p''(\lambda)}
			\leq
            \frac{z^3 \sup_{|w-\lambda| = 2r}|H_p(w)| }{(2r)^3(1-\tfrac{|z|}{2r})}
            \leq
			c_r \abs{z}^3,
		\end{equation*}
		where the bound follows from Cauchy's integral formula (\(z\mapsto H_p(\lambda + z)\) is analytic on a domain containing the closure of \(\bbD_{2r}\)).
		\item \label{LLT:prf:unif_second_derivatives_H_p_lambda} One has, with \(r\) still given by Observation~\ref{LLT:prf:analyticity_radius},
		\begin{equation*}
			\gamma
            =
            \inf_{p\in \calM}\inf_{\abs{\lambda}\leq r}H_p''(\lambda)
			>0.
		\end{equation*}
		This follows from the definition and compactness of \(\calM \times [-r,r]\), and from the irreducibility of \(a\). Indeed, \(H_p''(\lambda)\) is the variance of \(\frac{e^{\lambda x}}{M_p(\lambda)}p(x)\), which is controlled by \(a\) and the uniform bounds one has on \(\lambda\) and \(M_p(\lambda)\).
	\end{enumerate}
	Let then \(X_1,X_2,\dots\) be as in the statement. For \(i\geq 1\), denote
	\begin{equation*}
		M_i(z) = E(e^{zX_i}),\quad H_i(z) = \ln(M_i(z)),
		\quad \sigma_i^2 = \Var(X_i).
	\end{equation*}
	Also, for \(n\geq 1\), define
	\begin{equation*}
		\bbM_{n}(z) = \prod_{i=1}^{n}M_i(z),
		\quad
		\bbH_{n}(z) = \ln(\bbM_{n}(z)) = \sum_{i=1}^{n} H_i(z).
	\end{equation*}
	\begin{enumerate}[resume=LLTPrfObs]
		\item\label{LLT:prf:expansion_bbH_n} For any \(n\geq 1\), \(\lambda\in (-r,r)\), the function \(z\mapsto\bbH_{n}(\lambda + z)\) is
        well defined and analytic in \(\bbD_{3r}\). Moreover, by Observation~\ref{LLT:prf:expansion_H_p}, for \(z\in \bbD_{r}\),
		\begin{equation*}
			\Babs{\bbH_{n}(\lambda + z) - \bbH_{n}(\lambda) - \bbH_{n}'(\lambda) z - \frac{z^2}{2} \bbH_{n}''(\lambda)}
			\leq
			n c_r \abs{z}^3.
		\end{equation*}
		\item \label{LLT:prf:existence_tilt}Finally, as \(n^{\alpha} =o(n) \), there exists \(n_{\Delta}\geq 0\) (uniform over \(X_1,X_2,\dots\)) such that for any \(n\geq n_{\Delta}\), and any \(x\in \Z\) with \(\abs{x-m_n}\leq n^{\alpha}\), there exists \(\lambda\in (-r,r)\) such that \(\bbH_{n}'(\lambda) = x \), and \(\abs{\lambda} \leq  \frac{\abs{x-m_n}}{\gamma n}\). Indeed, for \(\abs{t}\leq r\), using Observation~\ref{LLT:prf:unif_second_derivatives_H_p_lambda}
		\begin{equation*}
			\babs{\bbH_n'(t) - \bbH'_n(0)}
			=
			\Babs{\int_{0}^{t} ds \bbH''_n(s)}
			\geq
			n \gamma \abs{t}.
		\end{equation*}
		Recalling that \(\bbH_n'(0) = m_n\), using the lower bound and the Intermediate Value Theorem gives the existence of the required \(\lambda\) as soon as \( n^{\alpha-1} \leq \gamma r\). Moreover, for the choice of \(\lambda\) corresponding to \(x\), one has \(\bbH_n'(\lambda) - \bbH'_n(0) = x-m_n\), giving the upper bound on \(\abs{\lambda}\).
	\end{enumerate}
	We are now ready to start the main part of the proof. For \(\lambda\in (-r,r)\), introduce \(X_{i}^{\lambda}\), \(i\geq 1\) an independent family of random variables with
	\begin{equation*}
		P(X_{i}^{\lambda} = x) = \frac{1}{M_i(\lambda)}p_i(x)e^{\lambda x},
	\end{equation*}
	and denote \(S_{k}^{\lambda}\) the associated random walk.
	
	Define
	\begin{equation*}
		n_0' = \max(n_{\Delta}, c_r^3).
	\end{equation*}
	with \(n_{\Delta}\) given in Observation~\ref{LLT:prf:existence_tilt}. Let now \(n\geq n_{0}'\), \(y\in \Z\) with \(\abs{y- m_n}\leq n^{\alpha}\). Let \(\lambda\in [-\frac{n^{\alpha}}{\gamma n},\frac{n^{\alpha}}{\gamma n}]\) be such that \(\bbH_n'(\lambda) = y\). Then, one has
	\begin{equation}
		\label{eq:LLT_prf:tilt}
		P(S_n = y)
		=
		\bbM_{n}(\lambda)E\Bigl(\mathds{1}_{S_{n}^{\lambda} = y} e^{-\lambda\sum_{i=1}^n X_i^{\lambda} }\Bigr)
		=
		e^{\bbH_n(\lambda)-\lambda y}P(S_{n}^{\lambda} = y).
	\end{equation}
	We then estimate the probability in the last expression. Note that the characteristic function of \(S_{n}^{\lambda}\) at \(\theta\) is given by \(\frac{\bbM_n(\lambda + \rmi \theta)}{\bbM_n(\lambda)}\).
	From the Fourier inversion formula, one has
	\begin{equation}
		\label{eq:LLT_prf:Fourier}
		P(S_{n}^{\lambda} = y)
		=
		\frac{1}{2\pi}\int_{-\pi}^{\pi} d\theta e^{-\rmi \theta y} 
		\frac{\bbM_n(\lambda + \rmi \theta)}{\bbM_n(\lambda)}.
	\end{equation}
	Letting \(\delta = \min(r,\pi, \tfrac{\gamma}{4 c_r})\), and \(\epsilon = \frac{1}{18}\), one can split the integral into three parts:
	\begin{equation*}
		\int_{-\pi}^{\pi} d\theta f(\theta)
		=
		\underbrace{\int_{\abs{\theta}<n^{-\frac{1}{2}+\epsilon}} d\theta f(\theta)}_{= I_{\epsilon}}
		+
		\underbrace{\int_{n^{-\frac{1}{2}+\epsilon}\leq \abs{\theta}< \delta} d\theta f(\theta)}_{= I_{\epsilon,\delta}}
		+
		\underbrace{\int_{\delta\leq \abs{\theta}\leq \pi} d\theta f(\theta)}_{= I_{\delta}},
	\end{equation*}
	with \(f(\theta) = e^{-\rmi \theta y} \frac{\bbM_n(\lambda + \rmi \theta)}{\bbM_n(\lambda)}\). To estimate \(I_{\delta}\), we use Observation~\ref{LLT:prf:norm_charac_funct}: one obtains
	\begin{equation}
		\label{eq:LLT_prf:I_delta}
		\abs{I_{\delta}}
		\leq
		\int_{\delta\leq \abs{\theta}\leq \pi} d\theta \prod_{i=1}^n\Babs{
			\frac{M_i(\lambda + \rmi \theta)}{M_i(\lambda)}}
		\leq
		2\pi e^{-c_{\delta} n},
	\end{equation}
	for some \(c_{\delta}>0\) depending only on \(\delta\) (and \(\delta_0,c_0,a\)).
	Then, we can use the expansion of Observation~\ref{LLT:prf:expansion_bbH_n} to handle \(I_{\epsilon,\delta}\):
	\begin{multline}
		\label{eq:LLT_prf:I_epsilon_delta}
		\abs{I_{\epsilon,\delta}}
		=
		\Babs{\int_{n^{-\frac{1}{2}+\epsilon}\leq \abs{\theta}< \delta} d\theta e^{-\rmi \theta y} 
			e^{\bbH_n(\lambda + \rmi \theta)-\bbH_n(\lambda)}}
		\leq
		\int_{n^{-\frac{1}{2}+\epsilon}\leq \abs{\theta}< \delta} d\theta 
		e^{ - \frac{\theta^2}{2}\bbH_n''(\lambda) + nc_r \abs{\theta}^3}
		\\
		\leq
		\int_{n^{-\frac{1}{2}+\epsilon}\leq \abs{\theta}< \delta} d\theta 
		e^{ - n\theta^2(\gamma - 2c_r \delta)/2}
		\leq
		2\int_{n^{-\frac{1}{2}+\epsilon}}^{\delta} d\theta 
		e^{ - \gamma n\theta^2/4}
		\leq
		2\pi e^{-\gamma n^{2\epsilon}/4},
	\end{multline}
	where we used Observation~\ref{LLT:prf:unif_second_derivatives_H_p_lambda}, \(\bbH_n'(\lambda) = y\), and the choice of \(\delta\) in the second line.
	Remains to control the leading term: \(I_{\epsilon}\). First, we again use the expansion of Observation~\ref{LLT:prf:expansion_bbH_n}
	\begin{equation}
		\label{eq:LLT_prf:I_epsilon_step1}
		I_{\epsilon}
		=
		\int_{-n^{-\frac{1}{2}+\epsilon}}^{n^{-\frac{1}{2}+\epsilon}} d\theta e^{-\rmi \theta y} 
		e^{\rmi \theta \bbH_n'(\lambda) - \frac{\theta^2}{2} \bbH_n''(\lambda)} e^{R(\theta)}
		=
		\frac{1}{\sqrt{n}}\int_{-n^{\epsilon}}^{n^{\epsilon}} d\phi 
		e^{- \frac{\phi^2}{2n} \bbH_n''(\lambda)} e^{R(\phi/\sqrt{n})}
	\end{equation}
	with \(\abs{R(\theta)}\leq n c_r \abs{\theta}^3\), where we used that by choice of \(\lambda\), \(\bbH_n'(\lambda) = y\), and we changed variable to \(\phi = \sqrt{n}\theta\). Then,
	\begin{equation*}
		\Babs{\int_{-n^{\epsilon}}^{n^{\epsilon}} d\phi 
			e^{- \frac{\phi^2}{2n} \bbH_n''(\lambda)} e^{R(\theta/\sqrt{n})}
			-
			\frac{\sqrt{2\pi}}{\sqrt{\bbH_n''(\lambda)/n}}}
		\leq
		\frac{\sqrt{2\pi} c_r e}{\sqrt{\gamma} }n^{-1/3}
		+
		\frac{2}{\gamma} n^{-\epsilon}e^{-\frac{\gamma}{2}n^{2\epsilon}}
	\end{equation*}
	as \(\abs{e^{R(\phi/\sqrt{n})}-1}\leq c_r e n^{-1/3}\) for \(\abs{\phi}\leq n^{\epsilon}\) (recall \(\epsilon = \frac{1}{18}\), and \(n\geq c_r^3\)),
	\(\int_{n^{\epsilon}}^{\infty} d\phi e^{-\frac{\phi^2}{2n}\bbH_n''(\lambda)} \leq \frac{1}{\gamma} n^{-\epsilon}e^{-\frac{\gamma}{2}n^{2\epsilon}} \),
	and \(\int_{-\infty}^{\infty} d\phi e^{- \frac{\phi^2}{2n}\bbH_n''(\lambda)} = \frac{\sqrt{2\pi}}{\sqrt{\bbH_n''(\lambda)/n}} \leq \frac{\sqrt{2\pi}}{\sqrt{\gamma}}\).
	Plugging this in~\eqref{eq:LLT_prf:I_epsilon_step1}, one gets
	\begin{equation}
		\label{eq:LLT_prf:I_epsilon}
		\Babs{I_{\epsilon} - \frac{\sqrt{2\pi}}{\sqrt{\bbH_n''(\lambda)}}} \leq \Cl{cst:LLT_1}n^{-5/6}
	\end{equation}
	where \(\Cr{cst:LLT_1}\) is some constant depending only on \(\gamma, c_r\). Injecting~\eqref{eq:LLT_prf:I_delta},~\eqref{eq:LLT_prf:I_epsilon_delta}, and~\eqref{eq:LLT_prf:I_epsilon} into~\eqref{eq:LLT_prf:Fourier}, one obtains
	\begin{equation}
		\label{eq:LLT_prf:tilted_proba_sharp}
		\Babs{P(S_{n}^{\lambda} = y) - \frac{1}{\sqrt{2\pi \bbH_n''(\lambda)}}}
		\leq \Cl{cst:LLT_2} n^{-5/6},
	\end{equation}
	for some \(\Cr{cst:LLT_2}\) depending only on \(\delta, c_r, \gamma\). Remains to estimate \(\bbH_n''(\lambda)\) and \(\bbH_n(\lambda)-\lambda y\) to go from~\eqref{eq:LLT_prf:tilt} to the theorem. As \(\abs{\lambda}\leq \frac{n^{\alpha}}{\gamma n}\), one has, using~\eqref{eq:boundH3},
	\begin{equation}
		\babs{\bbH_n''(\lambda)- \bbH_n''(0)} \leq 3c_r n \lambda \leq \tfrac{3c_r}{\gamma} n^{\alpha}.
	\end{equation}
	In particular, as \(\bbH_n''(0) = B_n\geq \gamma n\),
	\begin{equation}
		\label{eq:LLT_prf:tilted_to_non_tilted_polynomial_part}
		\Babs{\frac{\bbH_n''(\lambda)}{B_n} -1}
		\leq
		\tfrac{3c_r}{\gamma^2} n^{\alpha-1}
		\leq
		\Cl{cst:LLT_3}n^{-1/3},
	\end{equation}
	where \(\Cr{cst:LLT_3}\) is a constant depending only on \(c_r,\gamma\).
	Then,
	\begin{equation}
		\label{eq:LLT_prf:tilted_to_non_tilted_exp_part_H}
		\abs{\bbH_n(\lambda) - m_n \lambda - \tfrac{\lambda^2}{2} B_n}
		\leq
		\tfrac{c_r}{\gamma^3}n^{3\alpha-2}
	\end{equation}
	as \(\abs{\lambda}\leq \frac{1}{\gamma}n^{\alpha-1}\).
	Next, we have, with \(x= y-m_n\),
	\begin{equation*}
		\tfrac{\lambda^2}{2} B_n - \lambda x
		=
		\tfrac{B_n}{2}(\tfrac{x^2}{B_n^2} + 2 \Delta \tfrac{x}{B_n} + \Delta^2) - \tfrac{x^2}{B_n} - x \Delta
		=
		- \tfrac{x^2}{2B_n} + \tfrac{1}{2}\Delta^2 B_n,
	\end{equation*}
	with \(\Delta = \lambda - \frac{x}{B_n}\).
	Finally, as \(\bbH_n^{'}(\lambda) = y\), we get (again using Cauchy's integral formula to estimate the error term)
	\begin{equation*}
		x
		=
		\int_{0}^{\lambda} ds \bbH_n^{''}(s)
		=
		\lambda \bbH_n^{''}(0) +  \int_{0}^{\lambda} ds R(s) 
	\end{equation*}
	with \(\abs{R(s)}\leq c_r n \abs{s}\). In particular, (recall \(\bbH_n^{''}(0) = B_n\))
	\begin{equation*}
		\abs{\Delta} = \Babs{\frac{x}{B_n} - \lambda} \leq \frac{c_r n}{2B_n} \abs{\lambda}^2 \leq \frac{c_r}{2\gamma^2 } \frac{n^{2\alpha -1}}{B_n}.
	\end{equation*}
	So, \(\abs{\Delta^2 B_n}\leq  \Cl{cst:LLT_4} n^{4\alpha - 3}\) for some \(\Cr{cst:LLT_4}\) depending on \(\gamma, c_r\). Gathering everything from~\eqref{eq:LLT_prf:tilted_to_non_tilted_exp_part_H}, we get
	\begin{equation}
		\label{eq:LLT_prf:tilted_to_non_tilted_exp_part}
		\Babs{\bbH_n(\lambda)-\lambda y + \frac{x^2}{2B_n}}
		\leq
		\Cl{cst:LLT_5} n^{3\alpha-2}
	\end{equation}
	for some \(\Cr{cst:LLT_5}\) depending on \(\gamma,c_r\).
	
	We are now ready to conclude. From~\eqref{eq:LLT_prf:tilted_proba_sharp} and Observation~\ref{LLT:prf:unif_variances}, one has
	\begin{equation}
		\label{eq:LLT_prf:poly_part_final1}
		\babs{P(S_{n}^{\lambda} = y)\sqrt{2\pi \bbH_n''(\lambda)} - 1}
		\leq \Cl{cst:LLT_6} n^{-1/3},
	\end{equation}
	with \(\Cr{cst:LLT_6}\) depending only on \(\sigma_+,c_r,\gamma\). In particular, one has that by~\eqref{eq:LLT_prf:tilted_to_non_tilted_polynomial_part}
	\begin{equation}
		\label{eq:LLT_prf:poly_part_final}
		\babs{P(S_{n}^{\lambda} = y)\sqrt{2\pi B_n} - 1}
		\leq \Cl{cst:LLT_7} n^{-1/3},
	\end{equation}
	for some \(\Cr{cst:LLT_7}\) depending on \(\sigma_+,c_r,\gamma\). Plugging this and~\eqref{eq:LLT_prf:tilted_to_non_tilted_exp_part} in~\eqref{eq:LLT_prf:tilt} gives the claim.
\end{proof}

\section{Small ball for lattice bridges}
\label{sec:small_ball_lattice_bridges}

We now consider walk with steps in \(\calM_{\delta_0,c_0}^a\) and prove a refined version of the small-ball estimates at scale \(\sqrt{n}\) for bridges. We look at the probability that the walk stays close to the convex hull of its endpoints, which are allowed to be at distance \(\gg \sqrt{n}\) from one another. We chose to present the argument divided into several Lemmas, as we believe that the procedure can be useful in estimating more complicated events than small-ball probabilities.

\subsection{Coarse-Graining estimate}
\label{subsec:Gaussian_approx:CG_walks}

\begin{lemma}
	\label{lem:coarse_graining_walk}
	Let \(\delta_0,c_0\in (0,\infty)\), and \(a\in [0,1]^{\Z}\) be irreducible and aperiodic. Let \(\alpha\in [0,2/3)\). Then, there exist \(n_0\geq 1\), \(c',c, \epsilon> 0\), \(C\geq 0\) such that the following holds. For any \(n\geq n_0\), any independent sequence of random variables \(X_1,\dots,X_n\) all having law in \(\calM_{\delta_0,c_0}^a\), any \(x\in \Z-m_n\) with \(\abs{x}\leq n^{\alpha}\), and any \(K\geq 0\),
	\begin{equation*}
		P_0\bigl( \max_{k=1,\dots,n}\rmd(\bar{S}_k,[0,x]) \geq K \bgiven \bar{S}_n = x\bigr)
		\leq
		\begin{cases}
			Cn^{3/2}e^{c'x^2/n}e^{-c K^2/n} & \text{ if } K\leq \epsilon n,
			\\
			Cn^{3/2}e^{c'x^2/n}e^{ -c K} & \text{ if } K\geq \epsilon n,
		\end{cases}
	\end{equation*}
	where \(m_k = E_0(S_k)\).
\end{lemma}
\begin{proof}
	Use the shorthand \(P\equiv P_0\). Consider the case \(x\geq 0\). \(x\leq 0\) is treated in the exact same way. We start with a union bound:
	\begin{multline*}
		P\bigl( \max_{k=1,\dots,n}\rmd(\bar{S}_k,[0,x]) \geq K \bgiven \bar{S}_n = x\bigr)
		\\
		\leq
		\frac{1}{P(S_n = x+m_n)} \sum_{k=1}^{n} P\bigl( \{\bar{S}_k \leq -K\}\cup \{\bar{S}_n-\bar{S}_k \leq -K\} \bigr).
	\end{multline*}
	Then, one has from Theorem~\ref{thm:inhomo_LLT} that there are \(c>0,n_0'\geq 1\) uniform over the sequence \(X_1,\dots,X_n\) such that, if \(n\geq n_0'\),
	\begin{equation*}
		P(S_n=x+m_n)
		\geq
		\frac{c}{\sqrt{n}} e^{-x^2/B_n},
	\end{equation*}
	where \(B_k = \Var(S_k)\). Then,
	\begin{equation*}
		P\bigl( \{\bar{S}_k \leq -K\}\cup \{\bar{S}_n-\bar{S}_k \leq -K\} \bigr)
		\leq
		P( \bar{S}_k \leq -K) +P\bigl(\bar{S}_n-\bar{S}_k \leq -K \bigr).
	\end{equation*}
	To conclude the proof, use that, by Lemma~\ref{lem:very_large_tails_UB}, for any \(1\leq k\leq n\), and any sequence \(Y_1,\dots,Y_k\) with laws in \(\calM_{\delta_0,c_0}^a\)
	\begin{equation*}
		P\bigl( (\nu_1-Y_1) + \dots + (\nu_k-Y_k) \geq K\bigr)
		\leq
		\begin{cases}
			e^{-c' K^2/n} & \text{ if } K\leq \epsilon n,
			\\
			e^{-c' K} & \text{ if } K\geq \epsilon n,
		\end{cases}
	\end{equation*}
	where \(\nu_i = E(Y_i)\), and \(\epsilon,c'>0\) depend only on \(\delta_0,c_0,a\).
\end{proof}

\subsection{Gaussian Swapping}
\label{subsec:Gaussian_approx:Gaussian_swap}

We then prove an ``approximation by Gaussian'' result. For \(\alpha\geq 0\), \(n\geq 1 \), \(0=L_0<L_1<\dots<L_m = n\) introduce
\begin{equation}
	\label{eq:def:CtrlGrad}
	\ControlGrad_{n}^{\alpha}(L_0,\dots,L_m)
	=
	\big\{x\in \R^{n+1}:\ |x_{L_i}-x_{L_{i-1}}|\leq |L_i-L_{i-1}|^{\alpha}, i=1,\dots, m \big\}.
\end{equation}
\begin{lemma}
	\label{lem:Gaussian_approx_walk}
	With the same setup as Lemma~\ref{lem:coarse_graining_walk}, for any \(\alpha\in [0,2/3)\), there are \(n_0,c>0\), such that for any sequence \(X_1,\dots,X_n\) with laws in \(\calM_{\delta_0,c_0}^{a}\) the following holds. For any \(l\geq 1\), any \(L_0=0<L_1<\dots<L_l = n\) with
	\begin{equation*}
		L_i-L_{i-1}\geq n_0,\ i=1,\dots, l,
	\end{equation*}
	and any \(\epsilon_1,\dots \epsilon_m\in (0,1]\), one has that for every sets \(I_1\subset \Z -m_{L_1}, \dots,I_l \subset\Z-m_{L_l}\),
	\begin{align*}
		&\prod_{i=1}^l \tfrac{1}{\epsilon_i} e^{-c(L_i-L_{i-1})^{-\beta}}P\Bigl( \cap_{i=1}^{l}\{Z_{L_i}\in \bbI_i\} \cap \big\{Z\in\ControlGrad_{n}^{\alpha}(L_0,\dots,L_l) \big\} \Bigr)
		\\
		&\quad\leq
		P_0\Bigl( \cap_{i=1}^{l}\{\bar{S}_{L_i}\in I_i\} \cap \big\{\bar{S}\in\ControlGrad_{n}^{\alpha}(L_0,\dots,L_l) \big\} \Bigr)
		\\
		&\quad\quad\leq
		\prod_{i=1}^l \tfrac{1}{\epsilon_i} e^{c(L_i-L_{i-1})^{-\beta}} P\Bigl( \cap_{i=1}^{l}\{Z_{L_i}\in \bbI_i\} \cap \big\{Z\in\ControlGrad_{n}^{\alpha}(L_0,\dots,L_l) \big\} \Bigr)
	\end{align*}
	where \(\beta = \min(2-3\alpha,1/3)\), \(\bbI_i = \cup_{x\in I_i}[x-\epsilon_i/2,x+\epsilon_i/2]\), and \(Z\) is a random walk with independent increments of law \(Z_{k} -Z_{k-1}\sim \calN(0,\Var(X_k))\).
\end{lemma}
The \(\epsilon_i\)'s allow to take limits to deal with measures conditioned on the value of \(S_n, Z_n\).
\begin{figure}
	\centering
	\includegraphics{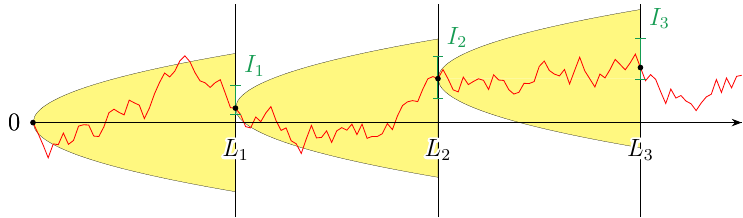}
	\caption{Construction in the proof of Lemma~\ref{lem:Gaussian_approx_walk}: The increments between times \(L_k\) and \(L_{k+1}\) have to lie in the yellow region.}
\end{figure}
\begin{proof}
	Let \(\alpha\in [0,2/3)\) and set \(\beta = \min(2-3\alpha,1/3)\). The claim follows straightforwardly from the Local Limit Theorem~\ref{thm:inhomo_LLT}, and the observation that for any \(A>0\), \(\epsilon>0\), \(x\in \R\),
	\begin{equation*}
		\frac{1}{\epsilon} e^{-\epsilon |x|/(2A)}\int_{x-\epsilon/2}^{x+\epsilon/2}du e^{-\frac{u^2}{2A}}
		\leq
		e^{-\frac{x^2}{2A}}
		\leq
		\frac{1}{\epsilon} e^{\epsilon |x|/(2A) + \epsilon^2/(8A)} \int_{x-\epsilon/2}^{x+\epsilon/2}du e^{-\frac{u^2}{2A}}.
	\end{equation*}
	Indeed, the constraint on the increments of \((\bar{S}_{L_0},\bar{S}_{L_1},\dots,\bar{S}_{L_l})\) gives that one can approximate the transitions \(P(\bar{S}_{L_i} = y \given \bar{S}_{L_{i-1}} = x)\), using Theorem~\ref{thm:inhomo_LLT}, by the probability for a Gaussian with variance \(\sum_{k=L_{i-1}+1}^{L_{i}} \Var(X_k)\) to fall into the interval \([y-x-\epsilon_i/2,y-x+\epsilon_i/2]\) up to an error factor lower, respectively upper, bounded by
	\begin{gather*}
		\tfrac{1}{\epsilon_i}e^{-c(L_i-L_{i-1})^{-\beta} - c'|x-y|\epsilon_i/(L_i-L_{i-1})}
		\geq
		\tfrac{1}{\epsilon_i}e^{-\tilde{c}(L_i-L_{i-1})^{-\beta}},
		\\
		\tfrac{1}{\epsilon_i}e^{c(L_i-L_{i-1})^{-\beta} + c'|x-y|\epsilon_i/(L_i-L_{i-1}) + c''\epsilon_i^2/(L_i-L_{i-1})}
		\leq
		\tfrac{1}{\epsilon_i}e^{\tilde{c}(L_i-L_{i-1})^{-\beta} },
	\end{gather*}
	as, by assumption, \(|x-y|\leq (L_{i}-L_{i-1})^{\alpha}\leq (L_{i}-L_{i-1})^{2/3}\).
\end{proof}

\subsection{Small ball estimates: Gaussian case}
\label{subsec:Gaussian_approx:Gaussian_walks}

\begin{theorem}
	\label{thm:Gaussian:small_ball_bridge}
	Let \(\sigma_+\in (0,+\infty)\). Then, for any \(n\geq 1\), any \(s>0\), any \(x\in \R\), any \(\sigma_1,\dots,\sigma_n\in [0,\sigma_+]\), and independent sequence \(X_1,\dots,X_n\) with \(X_i\sim \calN(0,\sigma_i^2)\),
	\begin{equation*}
		P_0\bigl(\max_{i=1,\dots,n} \rmd(S_i,[0,x]) \leq s\sqrt{n} \bgiven S_n= x \bigr)
		\geq
		\ThetaJac(s/\sigma_+)
	\end{equation*}
	where \(\ThetaJac(z) = \sum_{k\in \Z}(-1)^{k}e^{-2z^2k^2}>0\) for \(z>0\) is a Jacobi Theta Function. In particular, for any \(\epsilon>0\), there is \(z_{\epsilon}>0\) such that for any \(z\in (0,z_{\epsilon})\),
	\begin{equation*}
		\ThetaJac(z) \geq \exp(-(1+\epsilon)\tfrac{\pi^2}{8z^2}).
	\end{equation*}
\end{theorem}
\begin{proof}
	Let \(\sigma_+>0\). Let \(n\), \(x\), \(s\), \(\sigma_1,\dots,\sigma_n\) and \(X_1,\dots, X_n\) be as in the statement. Let \(x_- = 0\wedge x\), \(x_+ = 0\vee x\). Denote \(P_{0,x}=P_0(\cdot \given S_n = x)\), and \(E_{0,x}\) the associated expectation. Recall that \(B_k = \sum_{i=1}^k \sigma_i^2\), and let \(b_k = \frac{B_k}{B_n}\).
	Under \(P_{0,x}\), the vector \((S_0,S_1,\dots, S_n)\) is a Gaussian vector with mean and covariance, for \(0\leq i\leq j\leq n\),
	\begin{equation*}
		E_{0,x}(S_i ) = b_i x,
		\quad
		E_{0,x}\bigl((S_i-E_x(S_i))(S_j-E_x(S_j))\bigr) =  \frac{B_{i}(B_n -B_j)}{B_n},
	\end{equation*}
	where we used the formula for the conditional distribution of Gaussian vectors. In particular, \((S_0,\dots, S_n) \stackrel{\mathsf{law}}{=} (\phi_0, \phi_1 + b_1 x,\dots, \phi_n + x)\) where \((\phi_0 ,\dots, \phi_n)\) is a centred Gaussian vector with the same covariance as \(S\). The probability we want to estimate is then given by
	\begin{equation*}
		P\bigl( \phi_i + b_i x \in [x_- - s\sqrt{n}, x_+ + s\sqrt{n}] \ \forall i=0,\dots, n  \bigr).
	\end{equation*}
	But now, as \(0\leq b_i\leq 1\), one has
	\begin{equation*}
		\{\phi_i + b_i x \in [x_- - s\sqrt{n}, x_+ + s\sqrt{n}]\}
		\supset
		\{\phi_i \in [- s\sqrt{n}, s\sqrt{n}]\}.
	\end{equation*}
	In particular,
	\begin{multline*}
		P\bigl( \phi_i + b_i x \in [x_- - s\sqrt{n}, x_+ + s\sqrt{n}] \ \forall i=0,\dots, n  \bigr)
		\geq
		P\bigl( \abs{\phi_i} \leq s\sqrt{n} \ \forall i=0,\dots, n  \bigr).
	\end{multline*}
	Then, note that
	\begin{equation*}
		\frac{1}{\sqrt{B_n}}(\phi_0,\phi_1,\dots,\phi_n)
		\stackrel{\mathsf{law}}{=}
		(\BrownBri_{b_0}, \BrownBri_{b_1},\dots, \BrownBri_{b_n})
	\end{equation*}
	where \((\BrownBri_t)_{t\in [0,1]}\) is a standard Brownian Bridge. We get
	\begin{equation*}
		P\bigl( \abs{\phi_i} \leq s\sqrt{n} \ \forall i=0,\dots, n  \bigr)
		\geq
		P\Bigl(\sup_{t\in [0,1]} \abs{\BrownBri_t} \leq \frac{s}{\sigma_+}\Bigr),
	\end{equation*}
	since \(B_n \leq \sigma_+^2 n\) by our hypotheses on the \(\sigma_i\)'s. Finally, for any \(z>0\),
	\begin{equation*}
		P\Bigl(\sup_{t\in [0,1]} \abs{\BrownBri_t} \leq z\Bigr)
		=
		1 + 2\sum_{k=1}^{\infty} (-1)^k e^{-2k^2 z^2}
	\end{equation*}
	see~\cite[Equation (11.39)]{Billingsley-1968}. The asymptotic of this quantity as \(z\searrow 0\) is given by
	\begin{equation*}
		\lim_{z\to 0^+} z^2 \ln\Bigl(P\Bigl(\sup_{t\in [0,1]} \abs{\BrownBri_t} \leq z\Bigr)\Bigr) = -\frac{\pi^2}{8},
	\end{equation*}
	see~\cite[Theorem 6.3]{Li+Shao-2001} and references therein.
\end{proof}

\subsection{Small ball estimates: lattice case}
\label{subsec:Gaussian_approx:lattice_walks}

We start with the equivalent of Theorem~\ref{thm:Gaussian:small_ball_bridge} for walks.
\begin{theorem}
	\label{thm:Z_bridge:small_ball}
	Let \(\delta_0,c_0\in (0,+\infty)\), and \(a\in [0,1]^{\Z}\) be irreducible and aperiodic. Let \(s_0>0\), \(\alpha\in [0,2/3)\). There are \(c,c'>0, C\geq 0\), \(n_0\geq 1\), and \(\delta = \delta(\alpha)>0\) such that the following holds. For any \(s\geq s_0\), \(n\geq n_0\), any \(X_1,\dots,X_n\) independent sequence with laws in \(\calM_{\delta_0,c_0}^a\), and any \(x\in \Z-m_n\) with \(\abs{x}\leq n^{\alpha}\),
	\begin{equation*}
		P_0\bigl(\cap_{i=1}^{n} \bigl\{\rmd(\bar{S}_i,[0,x])\leq s\sqrt{n}\bigr\},\ \bar{S}_n = x\bigr)
		\geq
		\bigl(1-\tfrac{c}{n^{\delta}}\bigr)\ThetaJac( c' s)\frac{1}{\sqrt{2\pi B_n}}e^{-x^2/2B_n},
	\end{equation*}
	where \(m_n = E_0(S_n)\), and \(\ThetaJac\) is given in Theorem~\ref{thm:Gaussian:small_ball_bridge}.
\end{theorem}
\begin{proof}
	Use the shorthand \(P\equiv P_0\). Treat the case \(x\geq 0\). The case \(x\leq 0\) is handled in the same way. Suppose that we can fix \(\epsilon\in (0,\alpha), \alpha'\in (1/2,2/3)\) such that
    \begin{equation}
    \label{eq:prf:thm:Z_bridge:small_ball:exponents}
		(2\alpha'-1)(1-\epsilon) <\epsilon,
        \quad
        (1-\epsilon)\min(\tfrac{1}{3},2-3\alpha') >\epsilon,
        \quad
        \alpha - \epsilon < \alpha'(1-\epsilon).
	\end{equation}
	Let \(\calM \equiv \calM_{\delta_0,c_0}^a\). Let \(\ell = \lfloor n^{\epsilon}\rfloor\), and
	\begin{equation*}
		0=L_0<L_1<\dots <L_{\ell} = n,
		\quad
		\tfrac{1}{2}n^{1-\epsilon} \leq L_{i}-L_{i-1}\leq 2n^{1-\epsilon}.
	\end{equation*}
	Use the shorthand
	\begin{equation*}
		\ControlGrad \equiv \ControlGrad_{n}^{\alpha'}(L_1,\dots,L_{\ell}).
	\end{equation*}
	Then, by Lemma~\ref{lem:coarse_graining_walk}, for \(n\) large enough,
	\begin{multline*}
		P\bigl(\cap_{i=1}^{n} \bigl\{\rmd(\bar{S}_i,[0,x])\leq s\sqrt{n}\bigr\},\ \bar{S}_n = x\bigr)
		\geq
		(1-n^{\epsilon} e^{c'n^{(2\alpha'-1)(1-\epsilon)}-c s_0^2 n^{\epsilon}} )
		\\
		\cdot P\bigl(\bar{S}\in \ControlGrad,\ \cap_{j=1}^{\ell} \bigl\{\rmd(\bar{S}_{L_j},[0,x])\leq \tfrac{s}{2}\sqrt{n}\bigr\},\ \bar{S}_{L_{\ell}} = x\bigr)
		\\
		\geq
		(1-e^{-c n^{\epsilon}} )P\bigl(\bar{S}\in \ControlGrad,\ \cap_{j=1}^{\ell} \bigl\{\rmd(\bar{S}_{L_j},[0,x])\leq \tfrac{s}{2}\sqrt{n}\bigr\},\ \bar{S}_{L_{\ell}} = x\bigr)
        ,
	\end{multline*}
	for some \(c>0\) depending only on \(s_0,\delta_0,c_0,\alpha\), where we used~\eqref{eq:prf:thm:Z_bridge:small_ball:exponents}. Indeed, under \(\cap_{j=1}^{\ell} \bigl\{\rmd(\bar{S}_{L_j},[0,x])\leq \tfrac{s}{2}\sqrt{n}\bigr\}\), for \(\cap_{i=1}^{n} \bigl\{\rmd(\bar{S}_i,[0,x])\leq s\sqrt{n}\bigr\}\) not to be realized, there must be \(j\in \{1,\dots,\ell\}\) such that \(\rmd(\bar{S}_i,[\bar{S}_{L_{j-1}},\bar{S}_{L_{j}}])> \frac{s_0}{2}\sqrt{n}\). Using that under \(\ControlGrad_{n}^{\alpha'}(L_1,\dots,L_{\ell})\), \(|\bar{S}_{L_{j}}-\bar{S}_{L_{j-1}}|\leq 2n^{\alpha'(1-\epsilon)}\) gives the bound.
	Then, by Lemma~\ref{lem:Gaussian_approx_walk},
	\begin{multline*}
		P\bigl(\bar{S}\in\ControlGrad,\ \cap_{j=1}^{\ell} \bigl\{\rmd(\bar{S}_{L_j},[0,x])\leq \tfrac{s}{2}\sqrt{n}\bigr\},\ \bar{S}_{L_{\ell}} = x\bigr)
		\\
		\geq
		e^{-c n^{\epsilon-(1-\epsilon)\beta'}}
		P\bigl(Z\in \ControlGrad,\ \cap_{j=1}^{\ell} \bigl\{\rmd(Z_{L_j},[0,x])\leq \tfrac{s}{2}\sqrt{n}\bigr\},\ |Z_{L_{\ell}} -x|\leq \tfrac{1}{2}\bigr)
        ,
	\end{multline*}
	where \(\beta'=\min(1/3,2-3\alpha')\), and \(Z\) is a Gaussian walk with independent steps of law \(Z_{k}-Z_{k-1}\sim \calN(0,\Var(X_k))\), and \(c>0\) depends only on \(\delta_0,c_0,a\).
	
	We can now use Theorem~\ref{thm:Gaussian:small_ball_bridge} to obtain
	\begin{equation*}
		P\bigl(\cap_{j=1}^{\ell} \bigl\{\rmd(Z_{L_j},[0,x])\leq \tfrac{s}{2}\sqrt{n}\bigr\},\ |Z_{L_{\ell}}-x|\leq \tfrac{1}{2}\bigr)
		\geq
		\frac{e^{-x^2/2B_n}}{\sqrt{2\pi B_n}}\ThetaJac\bigl(\tfrac{s}{2\sigma_+}\bigr)
	\end{equation*}
	where \(\sigma_+ = \sup_{p\in \calM}\Var_p(X)\).
	Finally, by a union bound and large deviation estimates for Gaussians, for \(n\) large enough,
	\begin{align*}
		&P\bigl( \{Z\in \ControlGrad\}^c \bgiven Z_n = x\bigr)
		\leq
		\sum_{j=1}^{\ell} P\bigl( |Z_{L_j}-Z_{L_{j-1}}|\geq \tfrac{1}{2}n^{\alpha'(1-\epsilon)} \bgiven Z_n = x\bigr)
		\\
		&\qquad\leq
		\sum_{j=1}^{\ell} P\bigl( |Z_{L_j}-Z_{L_{j-1}}|\geq \tfrac{1}{2}n^{\alpha'(1-\epsilon)}-\tfrac{\sigma_+^2|x|(L_j-L_{j-1})}{\sigma_-^2 n} \bgiven Z_n = 0\bigr)
		\\
		&\qquad\leq
		\sum_{j=1}^{\ell}P\bigl( |Z_{L_j}-Z_{L_{j-1}}|\geq \tfrac{1}{3}n^{\alpha'(1-\epsilon)} \bgiven Z_n = 0\bigr)
		\\
		&\qquad\leq
		n^{\epsilon} \frac{6\sigma_+ \sqrt{2n^{1-\epsilon}}}{\sqrt{2\pi}n^{\alpha'(1-\epsilon)}} \exp(-\frac{n^{2\alpha'(1-\epsilon)}}{36\sigma_+^2 n^{1-\epsilon}})
		\leq
		e^{-c n^{(2\alpha'-1)(1-\epsilon)}},
	\end{align*}
	where \(c>0\), and we used that for any \(j=1,\dots,\ell_n\),
	\begin{equation*}
		\tfrac{1}{2}n^{\alpha'(1-\epsilon)}-\tfrac{\sigma_+^2|x|(L_j-L_{j-1})}{\sigma_-^2 n}
		\geq
		\tfrac{1}{2}n^{\alpha'(1-\epsilon)}-\tfrac{2\sigma_+^2}{\sigma_-^2}n^{\alpha-\epsilon}
        \geq
		\tfrac{1}{3}n^{\alpha'(1-\epsilon)},
	\end{equation*}
	which relies on~\eqref{eq:prf:thm:Z_bridge:small_ball:exponents}, and which holds for \(n\) larger than some \(\sigma_+/\sigma_-\)-dependent constant. One thus gets
	\begin{multline*}
		P\bigl( Z\in \ControlGrad,\ \cap_{j=1}^{\ell} \bigl\{\rmd(Z_{L_j},[0,x])\leq \tfrac{s}{2}\sqrt{n}\bigr\},\ |Z_{L_{\ell}}-x|\leq \tfrac{1}{2}\bigr)
		\\\geq
		\frac{e^{-x^2/2B_n}}{\sqrt{2\pi B_n}}\Bigl(\ThetaJac\bigl(\tfrac{s}{2\sigma_+}\bigr) - e^{-cn^{(2\alpha'-1)(1-\epsilon)}}\Bigr).
	\end{multline*}
	Gathering everything, and using that \(\ThetaJac\) is increasing, we get that for \(n\) large enough,
    \begin{multline*}
        P\bigl(\cap_{i=1}^{n} \bigl\{\rmd(\bar{S}_i,[0,x])\leq s\sqrt{n}\bigr\},\ \bar{S}_n = x\bigr)
		\\
        \geq
		(1-e^{-c n^{\epsilon}} )e^{-c n^{\epsilon-(1-\epsilon)\beta'}}
		\frac{e^{-x^2/2B_n}}{\sqrt{2\pi B_n}}\ThetaJac\bigl(\tfrac{s}{2\sigma_+}\bigr)\Bigl(1 - e^{-cn^{(2\alpha'-1)(1-\epsilon)}}\Bigr)
        \\
        \geq
		e^{-c n^{\epsilon-(1-\epsilon)\beta'}}
		\frac{e^{-x^2/2B_n}}{\sqrt{2\pi B_n}}\ThetaJac\bigl(\tfrac{s}{2\sigma_+}\bigr)
    \end{multline*}
    for some \(c\) depending only on \(\delta_0,c_0,a, s_0,\alpha\).
    We now find values of \(\epsilon,\alpha'\) satisfying~\eqref{eq:prf:thm:Z_bridge:small_ball:exponents}. Take
    \begin{equation*}
        (\alpha',\epsilon, \beta') = \begin{cases}
            (\tfrac{19}{36},\tfrac{1}{10},\tfrac{1}{3}) & \text{ if } \alpha \leq \tfrac{1}{2},
            \\
            (\tfrac{26}{45},\tfrac{19}{109},\tfrac{4}{15}) & \text{ if } \tfrac{1}{2} < \alpha \leq \tfrac{3}{5},
            \\
            (\tfrac{19}{36},\tfrac{55}{259},\tfrac{1}{3}) & \text{ if } \tfrac{3}{5} < \alpha \leq \tfrac{5}{8},
            \\
            (\tfrac{7-3\alpha}{15-10\alpha},\tfrac{8-7\alpha}{38-27\alpha},\tfrac{9-11\alpha}{15-10\alpha}) & \text{ if } \tfrac{5}{8} < \alpha < \tfrac{2}{3}.
        \end{cases}
    \end{equation*}These respectively give the values
    \begin{equation*}
        \delta(\alpha) = (1-\epsilon)\beta'-\epsilon =
        \begin{cases}
            \tfrac{1}{5} & \text{ if } \alpha \leq \tfrac{1}{2},
            \\
            \tfrac{5}{109} & \text{ if } \tfrac{1}{2} < \alpha \leq \tfrac{3}{5},
            \\
            \tfrac{13}{259} & \text{ if } \tfrac{3}{5} < \alpha \leq \tfrac{5}{8},
            \\
            \tfrac{10-15\alpha}{38-27\alpha} & \text{ if } \tfrac{5}{8} < \alpha < \tfrac{2}{3}.
        \end{cases}
    \end{equation*}
    and thus the claim.
\end{proof}

We then prove a ``small-ball with constrained endpoint'' result for ball size that can be mesoscopic or microscopic. The limitation will be that the excursion endpoint must live at the same scale as the smallness of the ball.
\begin{theorem}
	\label{thm:small_ball_fixed_endpoint}
	Let \(c_0,\delta_0>0\), \(a\in [0,1]^{\Z}\) be an irreducible, aperiodic sequence. Let \(\epsilon>0\). There are \(c_-,c_+,C\in (0,+\infty)\), \(\lambda_0\geq 0\), \(n_0\geq 1\), such that for any \(n\geq n_0\) and \(X_1,\dots, X_n\) independent sequence of random variables with laws in \(\calM_{\delta_0,c_0}^a\) one has the following.
	\begin{itemize}
		\item For any \(\sqrt{n}\geq \lambda\geq \lambda_0\), and any \(x\in \Z-m_n\) with \(|x|\leq (1-\epsilon)\lambda\),
		\begin{equation*}
			P_0\bigl(\max_{i=1,\dots,n} |\bar{S}_i| \leq \lambda,\ \bar{S}_n = x \bigr)
			\geq 
			\tfrac{1}{C\lambda}\exp(-\tfrac{c_- n}{\lambda^2}).
		\end{equation*}
		\item For any \(\sqrt{n}\geq \lambda\geq \lambda_0\), and any \(x\in \Z-m_n\) with \(|x|\leq \lambda\),
		\begin{equation*}
			P_0\bigl(\max_{i=1,\dots,n} |\bar{S}_i| \leq \lambda,\ \bar{S}_n = x \bigr)
			\leq 
			\tfrac{C}{\lambda}\exp(-\tfrac{c_+ n}{\lambda^2}).
		\end{equation*}
	\end{itemize}
\end{theorem}
\begin{proof}
	Use the shorthand \(P_0\equiv P\). Let \(L= \lceil \lambda^2 \rceil\). By Theorem~\ref{thm:Z_bridge:small_ball}, one has that for \(L\) large enough, \(|y|\leq \lambda/2\) and \(|x|\leq (1-\epsilon)\lambda\) in the support of \(\bar{S}_{n-L}\) and \(\bar{S}_n\) respectively,
	\begin{align*}
		P\bigl(\max_{i=n-L,\dots,n} |\bar{S}_i| \leq \lambda,\ \bar{S}_{n} = x \bgiven \bar{S}_{n-L} = y \bigr)
		&\geq
		\tfrac{1}{2}\ThetaJac(c'\epsilon) \frac{1}{\sqrt{2\pi B_{n-L+1,n}}} e^{-(x-y)^2/2B_{n-L+1,n}}
		\\
		&\geq
		c_{\epsilon} \tfrac{1}{\lambda},
	\end{align*}
	where \(B_{l,k} = \sum_{i= l}^k \Var(X_i)\). Thus, by Lemma~\ref{lem:micro_small_ball_walk_LB},
	\begin{align*}
		P\bigl(\max_{i=1,\dots,n} |\bar{S}_i| \leq \lambda,\ \bar{S}_n = x \bigr)
		&\geq
		c_{\epsilon} \tfrac{1}{\lambda} \sum_{|y|\leq \lambda/2} P\bigl(\max_{i=1,\dots,n-L} |\bar{S}_i| \leq \lambda,\ \bar{S}_{n-L} = y \bigr)
		\\
		&\geq
		c_{\epsilon} \tfrac{1}{\lambda} P\bigl(\max_{i=1,\dots,n-L} |\bar{S}_i| \leq \lambda,\ |\bar{S}_{n-L}| \leq \lambda/2 \bigr)
		\\
		&\geq
		C \tfrac{1}{\lambda} \exp(-c\tfrac{n}{\lambda^2})
	\end{align*}
	for some \(c,C>0\). This is the wanted lower bound. For the upper bound, with \(L\) as before and \(|x|\leq\lambda\), we have
	\begin{align*}
		P\bigl(\max_{i=1,\dots,n} |\bar{S}_i| \leq \lambda,\ \bar{S}_n = x \bigr)
		&\leq
		\sum_{|y|\leq \lambda} P\bigl(\max_{i=1,\dots,n-L} |\bar{S}_i| \leq \lambda,\ \bar{S}_{n-L} = y \bigr)P(\bar{S}_n = x \given \bar{S}_{n-L} = y)
		\\
		& \leq
		\tfrac{C}{\lambda} P\bigl(\max_{i=1,\dots,n-L} |\bar{S}_i| \leq \lambda\bigr)
		\leq
		\tfrac{C}{\lambda} \exp(-c\tfrac{n}{\lambda^2})
	\end{align*}
	where we used the Local Limit Theorem (Theorem~\ref{thm:inhomo_LLT}) in the second inequality, and Lemma~\ref{lem:micro_small_ball_walk_UB} in the third.
\end{proof}

\section{Positivity estimates: fixed endpoint}
\label{sec:fixed_endpoint}

Our final goal is to study bridges/excursions. We therefore need to prove versions of the results in Section~\ref{sec:traj_est_walks} for fixed endpoints. To this end, we need to restrict the class of steps considered (even to have a well defined and non-trivial measure under the endpoint constraint!). Our restriction on the steps are the same as the ones in Theorem~\ref{thm:inhomo_LLT}.

\subsection{Lower bound}
\label{subsec:positivity_fixed_endpoint:LB}

\begin{lemma}
\label{lem:positivity_fixed_endpoint:LB}
	Let \(c_0,\delta_0>0\), \(a\in [0,1]^{\Z}\) be an irreducible, aperiodic sequence. Let \(\alpha\in (1/2,2/3)\). There are \(C,c\in (0,+\infty)\), \(n_0\geq 1\), such that for any \(n\geq n_0\) and \(X_1,\dots, X_n\) independent sequence of random variables with laws in \(\calM_{\delta_0,c_0}^a\), one has the following. For any \(0\leq u,v\leq n^{\alpha}\) with \(P_u(\bar{S}_n = v)>0\),
	\begin{equation*}
		P_u\bigl(\bar{S}_n = v,\ \min_{i=1,\dots,n} \bar{S}_i \geq 0 \bigr)
		\geq
		\frac{C \min(u+1,\sqrt{n})\min(v+1,\sqrt{n})}{n^{3/2}}e^{-c(u-v)^2/n}.
	\end{equation*}
\end{lemma}
\begin{figure}
	\centering
	\includegraphics{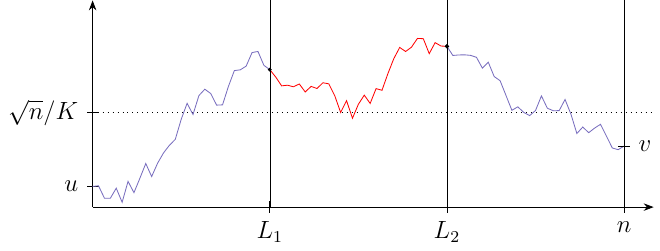}
	\caption{Lemma~6.1: Theorem~\ref{thm:Z_bridge:small_ball} can be used to estimate the contribution of the middle piece.}
\end{figure}
\begin{proof}
    Always assume \(n\) large enough so that everything works out.

	\medskip
    \noindent\textbf{Suppose first that \(\boldsymbol{ n^{\alpha} \geq u,v\geq \sqrt{n}}\).}
    Then, Theorem~\ref{thm:Z_bridge:small_ball} implies that
    \begin{equation*}
        P_u\bigl(\bar{S}_n = v,\ \min_{i=1,\dots,n} \bar{S}_i \geq 0 \bigr)
		\geq
        \frac{c}{\sqrt{n}}e^{-(u-v)^2/2B_n},
    \end{equation*}
    for some \(c>0\). This is the wanted claim in this case.
    
	\medskip
    \noindent\textbf{Suppose then that \(\boldsymbol{ 0\leq u,v\leq \sqrt{n}}\).}
    Let \(L_1 = \lceil \frac{n}{3} \rceil\), \(L_2 = \lfloor \frac{2n}{3}\rfloor\). Let \(K>0\) be large to be fixed later. By Theorem~\ref{thm:Z_bridge:small_ball}, for any \(x,y\in [\sqrt{n}/K, n^{\alpha}]\),
	\begin{equation*}
		P\bigl(S_{L_2} = y,\ \min_{i=L_1+1,\dots, L_2} S_i \geq 0 \given S_{L_1} = x\bigr)
		\geq
		\frac{c}{\sqrt{n}}
	\end{equation*}
	for some \(c>0\) depending only on \(K,c_0,\delta_0,a\). Thus, using Markov property,
	\begin{align*}
		&P_u\bigl(\bar{S}_n = v,\ \min_{i=1,\dots,n} \bar{S}_i \geq 0 \bigr)
		\\
        &\quad \geq
		\tfrac{c}{\sqrt{n}}\sum_{\sqrt{n}/K \leq x,y\leq K\sqrt{n}} P_u\bigl(\bar{S}_{L_1} = x,\ \min_{1,\dots,L_1} \bar{S}_i \geq 0 \bigr)
        P\bigl(\bar{S}_{n} = v,\ \min_{L_2,\dots,n} \bar{S}_i \geq 0 \bgiven S_{L_2} = y \bigr)
		\\
		&\quad=
		\tfrac{c}{\sqrt{n}} P_u\bigl(\bar{S}_{L_1} \in [\sqrt{n}/K, K\sqrt{n}],\ \min_{i=1,\dots,L_1} \bar{S}_i \geq 0 \bigr)
		\\
		&\quad\qquad \cdot P\bigl(\backvec{S}_{n-L_2} \in [\sqrt{n}/K, K\sqrt{n}],\ \min_{i=1,\dots,n-L_2} \backvec{S}_{i} \geq 0 \bgiven \backvec{S}_{0} = v \bigr),
	\end{align*}
	where \(\backvec{S}_{i} = \bar{S}_{n-i}\) is the time reversal of \(\bar{S}\). Now, on the one hand, by Lemma~\ref{lem:positivity_general_walks:moments}, and Chebychev's inequality,
	\begin{equation*}
		P_u\bigl(\bar{S}_{L_1} > K\sqrt{n} \bgiven \min_{i=1,\dots,L_1} \bar{S}_i \geq 0 \bigr)
		\leq
		\frac{c}{K^2}
	\end{equation*}
	for some \(c>0\) depending only on \(h,c_0,\delta_0,a\). On the other hand, by the CLT (Theorem~\ref{thm:inhomog_CLT}), for any \(n\geq n_0\) (with \(n_0\) depending only on \(c_0,\delta_0,a\))
	\begin{equation*}
		P_u\bigl(\bar{S}_{L_1} \geq \sqrt{n}/K \bgiven \min_{i=1,\dots,L_1} \bar{S}_i \geq 0 \bigr)
		\geq
		P_u\bigl(\bar{S}_{L_1} \geq \sqrt{n}/K\bigr)
		\geq
		\frac{1}{3}
	\end{equation*}
	where we used the FKG inequality in the first inequality, and \(K\) large and the CLT in the second. Taking \(K\) large enough, we get that
	\begin{equation*}
		P_u\bigl(\bar{S}_{L_1} \in [\sqrt{n}/K, K\sqrt{n}],\ \min_{i=1,\dots,L_1} \bar{S}_i \geq 0 \bigr)
		\geq
		\frac{1}{4} P_u\bigl(\min_{i=1,\dots,L_1} \bar{S}_i \geq 0 \bigr)
		\geq
		\frac{c(u+1)}{\sqrt{n}}
	\end{equation*}
	where we used Lemma~\ref{lem:positivity_general_walks:LB}. Proceeding similarly to bound the term involving the time-reversed walk, we obtain
	\begin{equation*}
		P_u\bigl(\bar{S}_n = v,\ \min_{i=1,\dots,n} \bar{S}_i \geq 0 \bigr)
		\geq
		\frac{c}{\sqrt{n}}\cdot \frac{c'(u+1)}{\sqrt{n}}\cdot \frac{c''(v+1)}{\sqrt{n}}
	\end{equation*}
	which is the claim in this case as \((u-v)^2 \leq n\).
    
    \medskip
    \noindent\textbf{Finally, suppose that \(\boldsymbol{ u\in [0,\sqrt{n}],\ v\in [\sqrt{n},n^{\alpha}] } \), or that \(\boldsymbol{ v\in [0,\sqrt{n}],\ u\in [\sqrt{n},n^{\alpha}] }\).}
    Both cases are treated the same way, as the second case is the time reverse of the first one. We thus only treat \(u\in [0,\sqrt{n}]\). Let \(L = \lfloor n/3\rfloor\). Let \(K>0\) to be fixed later. Then, for any \(x\in [\sqrt{n}/K, K\sqrt{n}]\), Theorem~\ref{thm:Z_bridge:small_ball} gives that,
	\begin{equation*}
		P\bigl(\bar{S}_{n} = v,\ \min_{i=L+1,\dots, n} S_i \geq 0 \given S_{L} = x\bigr)
		\geq
		\frac{C}{\sqrt{n}}e^{-cv^2/n},
	\end{equation*}
	for some \(C,c>0\) depending only on \(K,c_0,\delta_0,a\). We obtain
    \begin{align*}
		P_u\bigl(\bar{S}_n = v,\ \min_{i=1,\dots,n} \bar{S}_i \geq 0 \bigr)
        &\geq
		\frac{C}{\sqrt{n}}e^{-cv^2/n} \sum_{\sqrt{n}/K \leq x\leq K\sqrt{n}} P_u\bigl(\bar{S}_{L} = x,\ \min_{1,\dots,L} \bar{S}_i \geq 0 \bigr)
		\\
		&=
		\frac{C}{\sqrt{n}}e^{-cv^2/n} P_u\bigl(\bar{S}_{L} \in [\sqrt{n}/K, K\sqrt{n}],\ \min_{i=1,\dots,L_1} \bar{S}_i \geq 0 \bigr)
	    \\
        &\geq
		\frac{C(u+1)}{n}e^{-cv^2/n},
	\end{align*}
	as in the previous case. This is the claim in this last case as \(\frac{v^2}{n} \leq  \frac{(u-v)^2}{n}+1\).
\end{proof}

\subsection{Upper bound}
\label{subsec:positivity_fixed_endpoint:UB}

\begin{lemma}
\label{lem:positivity_fixed_endpoint:UB}
	Let \(c_0,\delta_0>0\), \(a\in [0,1]^{\Z}\) be an irreducible, aperiodic sequence. Let \(\alpha\in (1/2,2/3)\). There are \(c,C\in (0,+\infty)\), \(n_0\geq 1\), such that for any \(n\geq n_0\) and \(X_1,\dots, X_n\) independent sequence of random variables with laws in \(\calM_{\delta_0,c_0}^a\), one has the following. For any \(0\leq u,v\leq n^{\alpha}\),
	\begin{equation*}
		P_u\bigl(\bar{S}_n = v,\ \min_{i=1,\dots,n} \bar{S}_i \geq 0 \bigr)
		\leq
		\frac{C \min(u+1,\sqrt{n})\min(v+1,\sqrt{n})}{n^{3/2}}e^{-c(u-v)^2/n}.
	\end{equation*}
\end{lemma}
\begin{proof}
    Always assume \(n\) large enough so that everything works out. We do the same kind of case separation as in the proof of Lemma~\ref{lem:positivity_fixed_endpoint:LB}.

    \medskip
    \noindent\textbf{Suppose first that \(\boldsymbol{ n^{\alpha} \geq u,v\geq \sqrt{n}}\).}
    In this case,
    \begin{equation*}
        P_u\bigl(\bar{S}_n = v,\ \min_{i=1,\dots,n} \bar{S}_i \geq 0 \bigr)
        \leq
        P_u\bigl(\bar{S}_n = v\bigr).
    \end{equation*}
    The claim then follows from Theorem~\ref{thm:inhomo_LLT}.
    
    \medskip
    \noindent\textbf{Suppose then that \(\boldsymbol{ 0\leq u,v\leq 3\sqrt{n}}\).}
    Let \(L_1 = \lceil \frac{n}{3} \rceil\), \(L_2 = \lfloor \frac{2n}{3}\rfloor\). Then, by the CLT (Theorem~\ref{thm:inhomog_CLT}), for any \(x,y\), \(P(S_{L_2} = y\given S_{L_1} = x) \leq \frac{c}{\sqrt{n}}\) for some \(c>0\) depending only on \(c_0,\delta_0,a\). Then, by the Markov property
	\begin{align*}
		P_u\bigl(\bar{S}_n = v,\ \min_{i=1,\dots,n} \bar{S}_i \geq 0 \bigr)
		&\leq
		\sum_{x,y} P_u\bigl(\bar{S}_{L_1} = x,\ \min_{i=1,\dots,L_1} \bar{S}_i \geq 0 \bigr)P(S_{L_2} = y\given S_{L_1} = x)
		\\
		&\qquad \cdot P\bigl(\bar{S}_{n} = v,\ \min_{i=L_2+1,\dots,n} \bar{S}_i \geq 0 \bgiven S_{L_2} = y\bigr)
		\\
		&\leq
		\tfrac{c}{\sqrt{n}}\sum_{x,y} P_u\bigl(\bar{S}_{L_1} = x,\ \min_{i=1,\dots,L_1} \bar{S}_i \geq 0 \bigr)
		\\
		&\qquad \cdot P\bigl(\backvec{S}_{n-L_2} = y,\ \min_{i=1,\dots, n-L_2} \backvec{S}_{i} \geq 0 \bgiven \backvec{S}_{0} = v \bigr)
		\\
		&=
		\tfrac{c}{\sqrt{n}}P_u\bigl(\min_{i=1,\dots,L_1} \bar{S}_i \geq 0 \bigr)P\bigl(\min_{i=1,\dots, n-L_2} \backvec{S}_{i} \geq 0 \bgiven \backvec{S}_{0} = v \bigr)
		\\
		&\leq 
		\tfrac{c(u+1)(v+1)}{n^{3/2}}
	\end{align*}
	where \(\backvec{S}_{i} = \bar{S}_{n-i}\) is the time reversal of \(\bar{S}\), and we used Lemma~\ref{lem:positivity_general_walks:UB} in the last line. This is the claim in this case as \((u-v)^2\leq 9n\).
    
    \medskip
    \noindent\textbf{Finally, suppose that \(\boldsymbol{ u\in [0,\sqrt{n}],\ v\in [3\sqrt{n},n^{\alpha}] } \), or that \(\boldsymbol{ v\in [0,\sqrt{n}],\ u\in [3\sqrt{n},n^{\alpha}] }\).}
    By the same considerations as in the proof of Lemma~\ref{lem:positivity_fixed_endpoint:LB}, we can treat only the case \(u\in [0,\sqrt{n}]\). Define
    \begin{gather*}
        \tau' = \min\{k\in \{0,\dots, n\}:\ \bar{S}_k \geq \sqrt{n}\},
        \quad
        A = \{X_{\tau'} \leq \sqrt{n}\},
        \\
        \tau = \min\{k\geq 1:\ \bar{S}_k < 0\}.
    \end{gather*}
    Then, by a union bound and uniform exponential tails,
    \begin{equation*}
        P(A^c,\ \tau' \leq n)
        \leq
        ne^{-c\sqrt{n}}
        \leq
        \tfrac{C}{n}e^{-c v^2/n},
    \end{equation*}
    as \(v \leq n^{\alpha}\), \(\alpha\leq 2/3\).
    Then, as \(v> \sqrt{n}\),
    \begin{align*}
        &P_u\bigl(\bar{S}_n = v,\ \tau >n \bigr)
        \\
        &\quad \leq
        P(A^c,\ \tau' \leq n) +
        \sum_{k=0,\dots, n}\sum_{x\in [\sqrt{n},2\sqrt{n}]}P_u\bigl( \tau >k ,\ \tau'=k,\ \bar{S}_k = x\bigr) P( \bar{S}_n = v \given \bar{S}_k = x).
    \end{align*}
    Now, for \(x\in [\sqrt{n},2\sqrt{n}]\), recall that \(v\in [3\sqrt{n},n^{\alpha}]\) so that
    \begin{equation*}
        \sqrt{n}\leq v-x\leq n^{\alpha},
        \quad
        \tfrac{v^2}{9} \leq (v-x)^2 \leq v^2,
    \end{equation*}
    and thus \(P( \bar{S}_n = v \given \bar{S}_k = x)\) is upper bounded by
    \begin{equation*}
        \begin{cases}
            \frac{C}{\sqrt{n-k}}e^{-c(v-x)^2/(n-k)} \leq \frac{C}{\sqrt{n}}e^{-cv^2/n} & \text{ if } n-k\geq |v-x|^{1/\alpha},
            \\
            e^{-c(v-x)^2/(n-k)} \leq \tfrac{C}{\sqrt{n}}e^{-c v^2/n} & \text{ if }  \tfrac{1}{\rho}|v-x| \leq n-k < (v-x)^{1/\alpha} ,
            \\
            e^{-c|v-x|} \leq \frac{C}{\sqrt{n}}e^{-c v^2/n}  & \text{ if } n-k < \tfrac{1}{\rho}|v-x|,
        \end{cases}
    \end{equation*}
    where \(\rho >0 \) depends only on \(\delta_0,c_0,a\), and we used Theorem~\ref{thm:inhomo_LLT}, Lemma~\ref{lem:very_large_tails_UB}, the fact that \(r\mapsto re^{-cr^2}\) is upper bounded in the first case, the fact that in the second case \(\sqrt{n}e^{-cv^2/2(n-k)} \leq \sqrt{n}e^{-cv^{2-\frac{1}{\alpha}}} \leq \sqrt{n}e^{-cn^{\frac{2\alpha-1}{2\alpha}}} \leq C\) as \(\alpha>1/2\), and the fact that in the third case \(e^{-c|v-x|}\leq e^{-c\sqrt{n}}\leq \frac{C}{\sqrt{n}}e^{-cn^{2\alpha-1}} \leq \frac{C}{\sqrt{n}}e^{-cv^2/n}\) as \(\alpha < 2/3\). Plugging this in the previous bound, we get
    \begin{align*}
        &P_u\bigl(\bar{S}_n = v,\ \min_{i=1,\dots,n} \bar{S}_i \geq 0 \bigr)
        \\
        &\quad \leq
        \tfrac{C}{n}e^{-c v^2/n} +
        \tfrac{C}{\sqrt{n}}e^{-cv^2/n} \sum_{k=0,\dots, n}\sum_{x\in [\sqrt{n},2\sqrt{n}]}P_u\bigl(\tau >k,\ \tau'=k,\ \bar{S}_k = x\bigr)
        \\
        &\quad \leq
        \tfrac{C}{n}e^{-c v^2/n} +
        \tfrac{C}{\sqrt{n}}e^{-cv^2/n} P_u(\tau >\tau', \tau'\leq n).
    \end{align*}
    Remains to upper bound the probability in the last expression by \(\frac{C(u+1)}{\sqrt{n}}\). By Lemma~\ref{lem:positivity_general_walks:LB}, we have that for any \(k\in \{0,\dots,n\}\), and any \(x\geq \sqrt{n}\) with \(P_u(\bar{S}_k = x)>0\),
    \begin{equation*}
        P_u\bigl(\min_{i=k,\dots, n} \bar{S}_i\geq 0 \bgiven \bar{S}_k = x\bigr)
        \geq
        c>0,
    \end{equation*}
    where \(c\) depends only on \(\delta_0,c_0,a\). Then, using this and Markov's property,
    \begin{align*}
        P_u(\tau >\tau', \tau'\leq n)
        &\leq
        \sum_{k=0}^{n}\sum_{x\geq \sqrt{n}} P_u(\tau' = k, \bar{S}_k = x, \tau >k)\frac{P_u\bigl(\min_{i=k,\dots, n} \bar{S}_i\geq 0 \bgiven \bar{S}_k = x\bigr)}{c}
        \\
        &\leq
        C\sum_{k=0}^{n}\sum_{x\geq \sqrt{n}} P_u(\tau' = k,\ \bar{S}_k = x,\ \tau >n)
        \\
        &=
        CP_u(\tau' \leq n,\ \tau >n)
        \leq
        CP_u(\tau >n)
        \leq
        \frac{C(u+1)}{\sqrt{n}},
    \end{align*}
    by Lemma~\ref{lem:positivity_general_walks:UB}. This concludes the proof.
\end{proof}

\section{Trajectory estimates: excursions}
\label{sec:excursions}

\subsection{Small ball}
\label{subsec:excursions:small_ball}

\begin{lemma}
\label{lem:excursions:small_ball:LB}
    Let \(c_0,\delta_0>0\), \(a\in [0,1]^{\Z}\) be an irreducible, aperiodic sequence. Let \(K>0\). There are \(C,c\in (0,+\infty)\), \(\lambda_0,n_0\geq 0\), such that the following holds. For any \(n\geq n_0\), \(X_1,\dots, X_n\) independent sequence of random variables with laws in \(\calM_{\delta_0,c_0}^a\), any \(\lambda_0\leq \lambda \leq K\sqrt{n}\), and any \(0\leq u,v \leq \lambda\) with \(P_u(\bar{S}_n = v)>0\),
    \begin{multline*}
        P_u\bigl( 0\leq \min_{i=1,\dots,n} \bar{S}_i,\ \max_{i=1,\dots,n} \bar{S}_i\leq \lambda,\ \bar{S}_n = v \bigr)
        \\
        \geq
        \frac{C(\min(u, \lambda-u)+1)(\min(v, \lambda-v)+1)}{\lambda^3} \exp(-\tfrac{cn}{\lambda^2}).
    \end{multline*}
\end{lemma}
\begin{figure}
	\centering
	\includegraphics{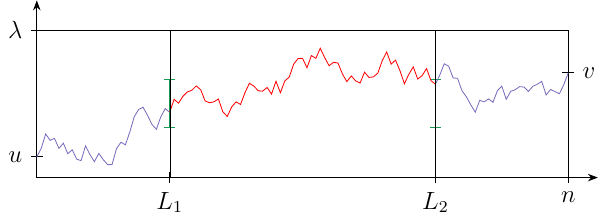}
	\caption{Proof of Lemma~\ref{lem:excursions:small_ball:LB}: wthe trajectory is forced to pass through the interval \([\tfrac13\lambda,\tfrac23\lambda]\) at times \(L_1\) and \(L_2\), allowing the use of Theorem~\ref{thm:small_ball_fixed_endpoint} to control the middle piece.}
\end{figure}
\begin{proof}
    Let \(L_1 = \lfloor\min( \lambda^2, n/3)\rfloor\), \(L_2 = \lceil\max( n-\lambda^2, 2n/3)\rceil\). Then, for any \(\frac{\lambda}{3}\leq x,y\leq \frac{2\lambda}{3}\) with \(P_u( \bar{S}_{L_1} = x, \bar{S}_{L_2} = y )>0\),
    \begin{equation*}
        P\bigl(0\leq \min_{i=L_1,\dots,L_2} \bar{S}_i,\ \max_{i=L_1,\dots,L_2} \bar{S}_i\leq \lambda,\ \bar{S}_{L_2} = y \bgiven \bar{S}_{L_1} = x\bigr)
        \geq
        \frac{C}{\lambda}\exp(-\tfrac{cn}{\lambda^2})
    \end{equation*}
    by Theorem~\ref{thm:small_ball_fixed_endpoint}. Thus, restricting to \(S_{L_1}, S_{L_2} \in [\lambda/3,2\lambda/3]\), and using the Markov property and the previous display,
    \begin{align*}
        &P_u\bigl( 0\leq \min_{i=1,\dots,n} \bar{S}_i,\ \max_{i=1,\dots,n} \bar{S}_i\leq \lambda,\ \bar{S}_n = v \bigr)
        \\
        &\qquad\geq
        \frac{C}{\lambda}\exp(-\tfrac{cn}{\lambda^2})
        P_u\bigl( 0\leq \min_{i=1,\dots,L_1} \bar{S}_i,\ \max_{i=1,\dots,L_1} \bar{S}_i\leq \lambda,\ 3\bar{S}_{L_1} \in [\lambda,2\lambda] \bigr)
        \\
        &\qquad\qquad\cdot P\bigl( 0\leq \min_{i=1,\dots,L_1} \backvec{S}_i,\ \max_{i=1,\dots,L_1} \backvec{S}_i\leq \lambda,\ 3\backvec{S}_{L_1} \in [\lambda,2\lambda] \bgiven \backvec{S}_{0} = v\bigr)
    \end{align*}
    where we introduced the time-reversed walk \(\backvec{S}_{i} = \bar{S}_{n-i}\). We now claim that
    \begin{equation*}
        P_u\bigl( 0\leq \min_{i=1,\dots,L_1} \bar{S}_i,\ \max_{i=1,\dots,L_1} \bar{S}_i\leq \lambda,\ 3\bar{S}_{L_1} \in [\lambda,2\lambda] \bigr)
        \geq
        \frac{C(\min(u,\lambda-u) +1)}{\lambda},
    \end{equation*}
    and similarly for the term involving the time-reversed walk, which would conclude the proof. We prove the bound for the walk, the time-reversed walk is treated the same way. Moreover, we assume that \(u\leq \lambda/2\), the other case being the same after a ceiling-to-floor change of point of view. Let \(\epsilon>0\) small to be fixed later, and take \(L = \lfloor \epsilon^2 \lambda^2\rfloor\).
    
    Then, for any \(x\in [\epsilon\lambda, (1-\epsilon)\lambda]\) with \(P_u(S_L = x)>0\) (and \(\lambda\) large enough), Theorem~\ref{thm:Z_bridge:small_ball} gives that
    \begin{equation*}
        P\bigl( 0\leq \min_{i=L,\dots,L_1} \bar{S}_i,\ \max_{i=L,\dots,L_1} \bar{S}_i\leq \lambda,\ 3\bar{S}_{L_1} \in [\lambda,2\lambda] \given S_L = x\bigr)
        \geq
        c,
    \end{equation*}
    where \(c>0\) depends only on \(\epsilon,c_0,\delta_0,a\). So, using Markov's property and inclusion of events,
    \begin{multline*}
        P_u\bigl( 0\leq \min_{i=1,\dots,L_1} \bar{S}_i,\ \max_{i=1,\dots,L_1} \bar{S}_i\leq \lambda,\ 3\bar{S}_{L_1} \in [\lambda,2\lambda] \bigr)
        \\
        \geq
        c P_u\bigl( 0\leq \min_{i=1,\dots,L} \bar{S}_i,\ \max_{i=1,\dots,L} \bar{S}_i\leq \lambda,\ \bar{S}_{L} \in [\epsilon\lambda,(1-\epsilon)\lambda] \bigr)
        \\
        \geq
        c P_u\bigl( 0\leq \min_{i=1,\dots,L} \bar{S}_i,\ \bar{S}_{L} \geq \epsilon \lambda,\ \max_{i=1,\dots,L} \bar{S}_i\leq (1-\epsilon)\lambda \bigr).
    \end{multline*}
    Now, first note that if \(u\geq 2\epsilon \lambda\), Lemma~\ref{lem:micro_small_ball_walk_LB} and inclusion of events gives
    \begin{equation*}
        P_u\bigl( 0\leq \min_{i=1,\dots,L} \bar{S}_i,\ \bar{S}_{L} \geq \epsilon \lambda,\ \max_{i=1,\dots,L} \bar{S}_i\leq (1-\epsilon)\lambda \bigr)
        \geq
        c>0,
    \end{equation*}
    which is the wanted bound in this case.
    So we consider the case \(u\leq 2\epsilon \lambda\). Then, the last probability is equal to
    \begin{equation*}
        P_u\bigl( \min_{i=1,\dots,L} \bar{S}_i \geq 0,\ \bar{S}_{L} \geq \epsilon \lambda \bigr)
        \bigl(1-
        P_u\bigl( \max_{i=1,\dots,L} \bar{S}_i> (1-\epsilon)\lambda \bgiven 0\leq \min_{i=1,\dots,L} \bar{S}_i,\ \bar{S}_{L} \geq \epsilon \lambda \bigr) \bigr).
    \end{equation*}
    Moreover, by Lemma~\ref{lem:positivity_fixed_endpoint:LB},
    \begin{equation*}
        P_u\bigl( \min_{i=1,\dots,L} \bar{S}_i \geq 0,\ \bar{S}_{L} \geq \epsilon \lambda \bigr)
        \geq \frac{c(u+1)}{\lambda}
    \end{equation*}
    for some \(c>0\). Finally, as \((\bar{S}_{i})_{i=0}^L\) is a submartingale under\linebreak \(P_u\bigl(\cdot \bgiven 0\leq \min_{i=1,\dots,L} \bar{S}_i,\ \bar{S}_{L} \geq \epsilon \lambda \bigr)\),
    \begin{equation*}
        P_u\bigl( \max_{i=1,\dots,L} \bar{S}_i> (1-\epsilon)\lambda \bgiven 0\leq \min_{i=1,\dots,L} \bar{S}_i,\ \bar{S}_{L} \geq \epsilon \lambda \bigr)
        \leq
        \frac{E_u\bigl( \bar{S}_L^2 \bgiven 0\leq \min_{i=1,\dots,L} \bar{S}_i,\ \bar{S}_{L} \geq \epsilon \lambda \bigr)  }{(1-\epsilon)^2\lambda^2}
    \end{equation*}
    by Doob's submartingale inequality. Finally, repeating the proof of Lemma~\ref{lem:positivity_general_walks:moments} in the present context,
    \begin{equation*}
        E_u\bigl( \bar{S}_L^2 \bgiven 0\leq \min_{i=1,\dots,L} \bar{S}_i,\ \bar{S}_{L} \geq \epsilon \lambda \bigr)
        \leq
        C(u^2 + L)
    \end{equation*}
    for some \(C\) depending only on \(\delta_0,c_0,a\). We have obtained
    \begin{equation}
        P_u\bigl( 0\leq \min_{i=1,\dots,L} \bar{S}_i,\ \bar{S}_{L} \geq \epsilon \lambda,\ \max_{i=1,\dots,L} \bar{S}_i\leq (1-\epsilon)\lambda \bigr)
        \geq
        \frac{c(u+1)}{\lambda}\Bigl(1-\frac{5C\epsilon^2}{(1-\epsilon)^2}\Bigr).
    \end{equation}
    Taking \(\epsilon>0\) small enough allows to conclude.
\end{proof}

\begin{lemma}
\label{lem:excursions:small_ball:UB}
    Let \(c_0,\delta_0>0\), \(a\in [0,1]^{\Z}\) be an irreducible, aperiodic sequence. There are \(c, C\in (0,+\infty)\), \(n_0\geq 1\), such that one has the following. For any \(n\geq n_0\), \(\sqrt{n}\geq \lambda\geq \lambda_0\), any \(X_1,\dots, X_n\) independent sequence of random variables with laws in \(\calM_{\delta_0,c_0}^a\), and any \(0\leq u,v \leq \lambda\),
    \begin{multline*}
        P_u\bigl( 0\leq \min_{i=1,\dots,n} \bar{S}_i,\ \max_{i=1,\dots,n} \bar{S}_i\leq \lambda,\ \bar{S}_n = v \bigr)
        \\
        \leq
        \frac{C(\min(u, \lambda-u)+1)(\min(v, \lambda-v)+1)}{\lambda^3} \exp(-\tfrac{cn}{\lambda^2}).
    \end{multline*}
\end{lemma}
\begin{proof}
    Let \(L_1 = \lfloor\min( \lambda^2, n/3)\rfloor\), \(L_2 = \lceil\max( n-\lambda^2, 2n/3)\rceil\). Then, for any \(x,y\in [0,\lambda]\), with \(P_u(\bar{S}_{L_1} = x) > 0\) and \(P(\bar{S}_{L_2}=y\given \bar{S}_{L_1} = x) > 0\),
    \begin{multline*}
        P\bigl(\bar{S}_{L_2}=y, 0\leq \min_{i=L_1,\dots,L_2} \bar{S}_i,\ \max_{i=L_1,\dots,L_2} \bar{S}_i\leq \lambda \bgiven \bar{S}_{L_1}= x \bigr)
        \\
        \leq
        P\bigl(\bar{S}_{L_2}=y,\ \max_{i=L_1,\dots,L_2} |\bar{S}_i| \leq 2\lambda \bgiven \bar{S}_{L_1}= x \bigr)
        \leq
        \frac{C}{\lambda}\exp(-\tfrac{cn}{\lambda^2})
    \end{multline*}
    by Theorem~\ref{thm:small_ball_fixed_endpoint}. Using Markov's property, the result will follow if we can prove that
    \begin{gather*}
        P_u\bigl(0\leq \min_{i=1,\dots,L_1} \bar{S}_i,\ \max_{i=1,\dots,L_1} \bar{S}_i\leq \lambda\bigr)
        \leq
        \frac{C(\min(u,\lambda-u) + 1)}{\lambda},
        \\
        P\bigl(0\leq \min_{i=1,\dots,n-L_2} \backvec{S}_i,\ \max_{i=1,\dots,n-L_2} \backvec{S}_i\leq \lambda \bgiven \backvec{S}_0 = v\bigr)
        \leq
        \frac{C(\min(v,\lambda-v) + 1)}{\lambda},
    \end{gather*}
    where \(\backvec{S}_k = \bar{S}_{n-k}\) is the time-reversed walk. We consider the first case, the second being treated in the exact same way. Moreover, a ceiling-to-floor change of point of view shows that it is enough to consider \(u\leq \lambda/2\) which is what we will do. Then,
    \begin{equation*}
        P_u\bigl(0\leq \min_{i=1,\dots,L_1} \bar{S}_i,\ \max_{i=1,\dots,L_1} \bar{S}_i\leq \lambda\bigr)
        \leq
        P_u\bigl(\min_{i=1,\dots,L_1} \bar{S}_i \geq 0\bigr)
        \leq 
        \frac{C(u+1)}{\lambda}
    \end{equation*}
    by Lemma~\ref{lem:positivity_general_walks:UB} and the definition of \(L_1\).
\end{proof}

\subsection{Excursions with a far away ceiling}
\label{subsec:excur_ceiling_far}

In this last section, we prove the results complementary to Lemmas~\ref{lem:excursions:small_ball:LB},~\ref{lem:excursions:small_ball:UB}. Namely, these Lemmas treated the case where the walk is forced to stay between level \(0\) and \(\lambda\) for \(\lambda \leq \sqrt{n}\), and here we treat the regime \(\lambda\geq\sqrt{n}\).

\begin{theorem}
	\label{thm:excursions:ceiling}
	Let \(c_0,\delta_0>0\), \(a\in [0,1]^{\Z}\) be an irreducible, aperiodic sequence. Let \(\alpha \in (0, 2/3)\). There are \(C_-,C_+,c_-,c_+\in (0,+\infty)\), \(n_0\geq 0\), such that the following holds. For any \(n\geq n_0\), \(X_1,\dots, X_n\) independent sequence of random variables with laws in \(\calM_{\delta_0,c_0}^a\), any \(\lambda \geq \sqrt{n}\), and any \(0\leq u,v \leq \lambda\) with \(P_u(\bar{S}_n = v)>0\) and \(|u-v|\leq n^{\alpha}\),
	\begin{multline*}
		P_u\bigl( 0\leq \min_{i=1,\dots,n} \bar{S}_i,\ \max_{i=1,\dots,n} \bar{S}_i\leq \lambda,\ \bar{S}_n = v \bigr)
		\\
		\geq
		\frac{C_-(\min(u, \lambda-u, \sqrt{n})+1)(\min(v, \lambda-v, \sqrt{n})+1)}{n^{3/2}} \exp(-\tfrac{c_-(u-v)^2}{n}),
	\end{multline*}
	and
	\begin{multline*}
		P_u\bigl( 0\leq \min_{i=1,\dots,n} \bar{S}_i,\ \max_{i=1,\dots,n} \bar{S}_i\leq \lambda,\ \bar{S}_n = v \bigr)
		\\
		\leq
		\frac{C_+(\min(u, \lambda-u, \sqrt{n})+1)(\min(v, \lambda-v, \sqrt{n})+1)}{n^{3/2}} \exp(-\tfrac{c_+(u-v)^2}{n}).
	\end{multline*}
\end{theorem}
\begin{proof}
	We treat separately different cases, and only sketch the proofs as the arguments are very similar to those we used in the rest of the paper. To shorten notations, introduce
	\begin{equation*}
		M^{\geq a}_{k,l} = \cap_{i=k}^l \{\bar{S_i}\geq a\},
		\quad
		M^{\leq a}_{k,l} = \cap_{i=k}^l \{\bar{S_i}\leq a\}.
	\end{equation*}
	
	\medskip
	\textbf{First consider \(u,v\in [\frac{\sqrt{n}}{10}, \lambda -\frac{\sqrt{n}}{10}]\).} There, both bounds follow from Theorems~\ref{thm:Z_bridge:small_ball} and~\ref{thm:inhomo_LLT} and inclusion of events.
	
	\medskip
	\textbf{Second, consider either of the cases \(u\leq \frac{\sqrt{n}}{10}, v \in [\frac{\sqrt{n}}{10}, \lambda -\frac{\sqrt{n}}{10} ]\), or \(u\geq \lambda- \frac{\sqrt{n}}{10}, v \in [\frac{\sqrt{n}}{10},\lambda -\frac{\sqrt{n}}{10} ]\), or one of the two with the roles of \(u,v\) interchanged.}
	There, we can restrict to the first two cases by considering the time-reversed walk, and to the first case by ceiling-to-floor change of viewpoint. We start with the lower bound by setting \(L = \lfloor \frac{n}{2}\rfloor\) and forcing \(\bar{S}_L\in [\frac{\sqrt{n}}{10}, \frac{2\sqrt{n}}{10}]\).
	\begin{multline*}
		P_u\bigl( M_{1,n}^{\geq 0},\ M_{1,n}^{\leq \lambda},\ \bar{S}_n = v \bigr)
		\\
		\geq
		\sum_{\frac{\sqrt{n}}{10}\leq x\leq \frac{2\sqrt{n}}{10} } P_u\bigl( M_{1,L}^{\geq 0},\ M_{1,L}^{\leq \lambda},\ \bar{S}_L = x \bigr) P\bigl( M_{L,n}^{\geq 0},\ M_{L,n}^{\leq \lambda}, \bar{S}_n = v \bgiven \bar{S}_L = x \bigr).
	\end{multline*}
	Now, for \(x\) as in the above sum,
	\begin{gather*}
		P_u\bigl( M_{1,L}^{\geq 0},\ M_{1,L}^{\leq \lambda},\ \bar{S}_L = x \bigr)
		\geq
		\tfrac{C (u+1)}{n},
		\\
		P\bigl( M_{L,n}^{\geq 0},\ M_{L,n}^{\leq \lambda}, \bar{S}_n = v \bgiven \bar{S}_L = x \bigr)
		\geq
		\tfrac{C}{\sqrt{n}}\exp(-c\tfrac{(v-x)^2}{n}),
	\end{gather*}
	by Lemma~\ref{lem:excursions:small_ball:LB} and inclusion of events, and by Theorem~\ref{thm:Z_bridge:small_ball} respectively. Summing over \(x\) then yields the wanted bound. Now, for the upper bound, we use inclusion of events and Lemma~\ref{lem:positivity_fixed_endpoint:UB}:
	\begin{equation*}
		P_u\bigl( M_{1,n}^{\geq 0},\ M_{1,n}^{\leq \lambda},\ \bar{S}_n = v \bigr)
		\leq
		P_u\bigl( M_{1,n}^{\geq 0},\ \bar{S}_n = v \bigr)
		\leq
		\tfrac{C(u+1)}{n}\exp(-c(u-v)^2/n),
	\end{equation*}
	which is the wanted upper bound in this case.
	
	\medskip
	\textbf{Third, consider either of the cases \(u\leq \frac{\sqrt{n}}{10}, v \leq \frac{\sqrt{n}}{10}\), or \(u\geq \lambda- \frac{\sqrt{n}}{10}, v \geq \lambda- \frac{\sqrt{n}}{10}\).}
	The two are related to each other by a ceiling-to-floor change of viewpoint, so we only consider the first case.
	Start with the lower bound. We get a lower bound by inclusion of events and Lemma~\ref{lem:excursions:small_ball:LB}:
	\begin{equation*}
		P_u\bigl( M_{1,n}^{\geq 0},\ M_{1,n}^{\leq \lambda},\ \bar{S}_n = v \bigr)
		\geq
		P_u\bigl( M_{1,n}^{\geq 0},\ M_{1,n}^{\leq \sqrt{n}},\ \bar{S}_n = v \bigr)
		\geq
		\tfrac{C(u+1)(v+1)}{n^{3/2}},
	\end{equation*}
	which is the wanted lower bound in this case. Now, the upper bound follows from inclusion of events and Lemma~\ref{lem:positivity_fixed_endpoint:UB}:
	\begin{equation*}
		P_u\bigl( M_{1,n}^{\geq 0},\ M_{1,n}^{\leq \lambda},\ \bar{S}_n = v \bigr)
		\leq
		P_u\bigl( M_{1,n}^{\geq 0},\ \bar{S}_n = v \bigr)
		\leq
		\tfrac{C(u+1)(v+1)}{n^{3/2}},
	\end{equation*}
	which is the wanted upper bound in this case.
	
	\medskip
	\textbf{Finally, consider either of the cases \(u\leq \frac{\sqrt{n}}{10}, v \geq \lambda-\frac{\sqrt{n}}{10}\), or \(u\geq \lambda- \frac{\sqrt{n}}{10}, v \leq \frac{\sqrt{n}}{10}\).}
	The two are related to each other by a time reversal of the walk, so we only consider the first case.
	Let \(L_1 = \lfloor \frac{n}{3} \rfloor\), \(L_2 = \lceil \frac{2n}{3}\rceil\). Start with the lower bound. We get a lower bound by forcing \(\bar{S}_{L_1}\in [\tfrac{\sqrt{n}}{10},\tfrac{2\sqrt{n}}{10}]\), and \(\bar{S}_{L_2}\in [\lambda-\tfrac{2\sqrt{n}}{10},\lambda- \tfrac{\sqrt{n}}{10}]\). We get
	\begin{multline*}
		P_u\bigl( M_{1,n}^{\geq 0},\ M_{1,n}^{\leq \lambda},\ \bar{S}_n = v \bigr)
		\geq
		\sum_{\frac{\sqrt{n}}{10}\leq x\leq \frac{2\sqrt{n}}{10} }\sum_{\lambda-\frac{2\sqrt{n}}{10}\leq y\leq \lambda - \frac{\sqrt{n}}{10} } P_u\bigl( M_{1,L_1}^{\geq 0},\ M_{1,L_1}^{\leq \lambda},\ \bar{S}_{L_1} = x \bigr)
		\\
		\cdot P\bigl( M_{L_1,L_2}^{\geq 0},\ M_{L_1,L_2}^{\leq \lambda}, \bar{S}_{L_2} = y \bgiven \bar{S}_{L_1} = x \bigr) P\bigl( M_{L_2,n}^{\geq 0},\ M_{L_2,n}^{\leq \lambda}, \bar{S}_{n} = v \bgiven \bar{S}_{L_2} = y \bigr).
	\end{multline*}
	Now, for \(x,y\) as in the above sum,
	\begin{gather*}
		P_u\bigl( M_{1,L_1}^{\geq 0},\ M_{1,L_1}^{\leq \lambda},\ \bar{S}_{L_1} = x \bigr)
		\geq
		\tfrac{C (u+1)}{n},
		\\
		P\bigl( M_{L_2,n}^{\geq 0},\ M_{L_2,n}^{\leq \lambda},\ \bar{S}_{n} = v\bgiven \bar{S}_{L_2} = y \bigr)
		\geq
		\tfrac{C (\lambda - v + 1)}{n},
		\\
		P\bigl( M_{L_1,L_2}^{\geq 0},\ M_{L_1,L_2}^{\leq \lambda}, \bar{S}_{L_2} = y \bgiven \bar{S}_{L_1} = x \bigr)
		\geq
		\tfrac{C}{\sqrt{n}}\exp(-c\tfrac{(v-u)^2}{n}),
	\end{gather*}
	by Lemma~\ref{lem:excursions:small_ball:LB} and inclusion of events for the first two, and by Theorem~\ref{thm:Z_bridge:small_ball} and the observation \(|u-v| - \frac{\sqrt{n}}{5}\leq |y-x|\leq |u-v| \) for the last. Plugging these bounds in the previous display and summing, we obtain the wanted lower bound. Remains to prove the upper bound. We use a similar decomposition and inclusion of events to get
	\begin{multline*}
		P_u\bigl( M_{1,n}^{\geq 0},\ M_{1,n}^{\leq \lambda},\ \bar{S}_n = v \bigr)
		\leq
		\sum_{0\leq x,y\leq \lambda } P_u\bigl( M_{1,L_1}^{\geq 0},\ \bar{S}_{L_1} = x \bigr)
		\\
		\cdot P\bigl(\bar{S}_{L_2} = y \bgiven \bar{S}_{L_1} = x \bigr) P\bigl(M_{L_2,n}^{\leq \lambda}, \bar{S}_{n} = v \bgiven \bar{S}_{L_2} = y \bigr).
	\end{multline*}
	We then bound the first and last probabilities using Lemma~\ref{lem:positivity_fixed_endpoint:UB}, and the second one using Theorem~\ref{thm:inhomo_LLT} to obtain
	\begin{multline*}
		P_u\bigl( M_{1,n}^{\geq 0},\ M_{1,n}^{\leq \lambda},\ \bar{S}_n = v \bigr)
		\\\leq
		\frac{C(u+1)(\lambda-v+1)}{n^{5/2}} \sum_{0\leq x,y\leq \lambda } \exp(-c\bigl((u-x)^2 + (x-y)^2 + (v-y)^2 \bigr)/n )
		\\\leq
		\frac{C(u+1)(\lambda-v+1)}{n^{5/2}} \sum_{0\leq x,y\leq \lambda } \exp(-c\bigl(x^2 + (x-y)^2 + (y-\lambda)^2 \bigr)/n ),
	\end{multline*}
	where we used the constraints on \(u,v\) in the last line. A simple sum-integral comparison gives
	\begin{multline*}
		\sum_{0\leq x,y\leq \lambda } \exp(-c\bigl(x^2 + (x-y)^2 + (y-\lambda)^2 \bigr)/n )
		\\
		\leq
		C \int_{0}^{\infty} dx \int_{0}^{\infty} dy \exp(-c\bigl(x^2 + (x-y)^2 + (\lambda-y)^2 \bigr)/n )
		\leq
		Cn\exp(-c\tfrac{\lambda^2}{n}).
	\end{multline*}
	Using \( \lambda - \tfrac{2\sqrt{n}}{10}\leq |v-u|\leq \lambda\), we obtain the wanted upper bound in the last case.
\end{proof}

\subsection{Tails}
\label{subsec:excursions:tails}
\begin{lemma}
    \label{lem:tails}
    Let \(c_0,\delta_0>0\), \(a\in [0,1]^{\Z}\) be an irreducible, aperiodic sequence. Let \(\beta\in (0,1/6)\). There are \(n_0,t_0\geq 0\), \(C,c\in (0,+\infty)\) such that the following holds. For any \(n\geq n_0\), any \(X_1,\dots, X_n\) independent sequence of random variables with laws in \(\calM_{\delta_0,c_0}^a\), and any \(0\leq u,v \leq \tfrac12t\sqrt{n}\) with \(P_u(\bar{S}_n = v)>0\),
    \begin{equation*}
        P_u\bigl(\min_{i=1,\dots,n} \bar{S}_i\geq 0,\ \bar{S}_{k}\geq t\sqrt{n},\ \bar{S}_n = v\bigr)
        \begin{cases}
            \geq \frac{\min(u+1,\sqrt{n})\min(v+1,\sqrt{n})}{Ctn^{3/2}}\exp(-c t^2),
            \\
            \leq \frac{C\min(u+1,\sqrt{n})\min(v+1,\sqrt{n})}{tn^{3/2}}\exp(- t^2/c),
        \end{cases}
    \end{equation*}
    for all \(n/3\leq k\leq 2n/3\), and all \(t_0\leq t\leq n^{\beta}\).
\end{lemma}
\begin{proof}
Suppose \(t\geq 2\). We treat the upper and lower bound separately. Let \(\alpha = \frac{1}{2}+\beta\in (1/2,2/3)\). Introduce
\begin{equation*}
    D = \{\min_{i=1,\dots,n} \bar{S}_i\geq 0\}.
\end{equation*}

\medskip
\noindent\textbf{We start with the upper bound.}
First,
\begin{equation*}
    P_u\bigl(D,\ \bar{S}_{k}\geq t\sqrt{n},\ \bar{S}_n = v\bigr)
    \leq
    P_u\bigl( \bar{S}_{k}\geq n^{\alpha}\bigr)
    +
    \sum_{t\sqrt{n}\leq x\leq n^{\alpha}} P_u\bigl(D,\ \bar{S}_{k}=x,\ \bar{S}_n = v\bigr).
\end{equation*}
Then, by Lemma~\ref{lem:very_large_tails_UB}, and the fact that \(2\alpha-1 = 2\beta>0\),
\begin{equation*}
    P_u\bigl( \bar{S}_{k}\geq n^{\alpha}\bigr)
    \leq
    e^{-cn^{2\alpha -1}}
    =
    e^{-cn^{2\beta}}
    \leq
    \tfrac{C\min(u+1,\sqrt{n})\min(v+1,\sqrt{n})}{n^{3/2}}e^{-ct^{2}}.
\end{equation*}
Now, using Lemma~\ref{lem:positivity_fixed_endpoint:UB}, for \(t\sqrt{n}\leq x\leq n^{\alpha}\)
\begin{align*}
    &P_u\bigl(D,\ \bar{S}_{k}=x,\ \bar{S}_n = v\bigr)
    \\
    &\quad=
    P_u\bigl(\min_{i=1,\dots,k}\bar{S}_i \geq 0,\ \bar{S}_{k}=x\bigr)P\bigl(\min_{i=1,\dots,n-k}\backvec{S}_i \geq 0,\ \backvec{S}_{n-k}=x \bgiven \backvec{S}_0 = v\bigr)
    \\
    &\quad\leq
    \tfrac{C \min(u+1,\sqrt{n})}{n}e^{-c(u-x)^2/n} \tfrac{C \min(v+1,\sqrt{n})}{n}e^{-c(x-v)^2/n}
    \\
    &\quad\leq
    \tfrac{C \min(u+1,\sqrt{n})\min(v+1,\sqrt{n})}{n^2}e^{-cx^2/n},
\end{align*}
where \(\backvec{S}_i = \bar{S}_{n-i}\) is the time-reversed walk, and we used \(\frac{x}{2}\leq x-u, x-v\leq x \), for the concerned \(x\)'s as \(t\geq 2\) and \(0\leq u,v\leq \tfrac12t\sqrt{n}\leq \tfrac12 x\). Now, by a Riemann sum approximation,
\begin{equation*}
    \sum_{t\sqrt{n}\leq x\leq n^{\alpha}} \tfrac{1}{\sqrt{n}}e^{-cx^2/n}
    \leq
    C\int_{t}^{\infty}dx e^{-cx^2} \leq \tfrac{C}{t}e^{-ct^2}.
\end{equation*}
This concludes the upper bound.

\medskip
\noindent\textbf{We then prove the lower bound.}
We have
\begin{equation*}
    P_u\bigl(D,\ \bar{S}_{k}\geq t\sqrt{n},\ \bar{S}_n = v\bigr)
    \geq
    \sum_{t\sqrt{n}\leq x\leq n^{\alpha}} P_u\bigl(D,\ \bar{S}_{k}=x,\ \bar{S}_n = v\bigr).
\end{equation*}
Then, using Lemma~\ref{lem:positivity_fixed_endpoint:LB}, for \(t\sqrt{n}\leq x\leq n^{\alpha}\),
\begin{align*}
    &P_u\bigl(D,\ \bar{S}_{k}=x,\ \bar{S}_n = v\bigr)
    \\
    &\quad=
    P_u\bigl(\min_{i=1,\dots,k}\bar{S}_i \geq 0,\ \bar{S}_{k}=x\bigr)P\bigl(\min_{i=1,\dots,n-k}\backvec{S}_i \geq 0,\ \backvec{S}_{n-k}=x \bgiven \backvec{S}_0 = v\bigr)
    \\
    &\quad\geq
    \tfrac{C \min(u+1,\sqrt{n})}{n}e^{-c(u-x)^2/n} \tfrac{C \min(v+1,\sqrt{n})}{n}e^{-c(x-v)^2/n}
    \\
    &\quad\geq
    \tfrac{C \min(u+1,\sqrt{n})\min(v+1,\sqrt{n})}{n^2}e^{-cx^2/n},
\end{align*}
where \(\backvec{S}_i = \bar{S}_{n-i}\) is again the time-reversed walk, and we used \(\frac{x}{2}\leq x-u, x-v\leq x \), for the concerned \(x\)'s as \(t\geq 2\) and \(0\leq u,v\leq \tfrac12t\sqrt{n}\leq \tfrac12 x\). Now, by a Riemann sum approximation,
\begin{equation*}
    \sum_{t\sqrt{n}\leq x\leq n^{\alpha}} \tfrac{1}{\sqrt{n}}e^{-cx^2/n}
    \geq
    C\int_{t}^{\infty}dx e^{-cx^2} \geq \tfrac{C}{t+t^{-1}}e^{-ct^2}.
\end{equation*}
This concludes the lower bound.
\end{proof}

\section*{Acknowledgements}

S.O.\ thanks Quentin Berger for showing him the elegant martingale argument of appendix~\ref{app:bounded_increm_pos_proba} which served as the basis for the proofs of Lemmas~\ref{lem:positivity_general_walks:LB} and~\ref{lem:positivity_general_walks:UB}. Y.V. is partially supported by the Swiss NSF through the NCCR SwissMAP.

\appendix

\section{Simple proof in the bounded case}
\label{app:bounded_increm_pos_proba}

\begin{theorem}
	\label{thm:bounded_increm_pos_proba}
	Let \(K>0\), and \(\alpha>0\). Then, there are \(c_+,c_-, n_0>0\) such that the following holds. For any \(n\geq n_0\), and any independent sequence of real random variables \(X_1,X_2,\dots\) satisfying that for all \(i\)'s
	\begin{equation*}
		E(X_i) = 0,\ 
		E(X_i\mathds{1}_{X_i>0})\geq \alpha,\ 
		|X_i| \leq K\ \text{ a.s.},
	\end{equation*}
	one has that for any \(u<\sqrt{n}\),
	\begin{equation*}
		\frac{c_-(1+u)}{\sqrt{n}} \leq P\bigl(\min_{i=1,\dots, n} S_i >0 \bgiven S_0 = u\bigr) \leq \frac{c_+(1+u)}{\sqrt{n}}.
	\end{equation*}
\end{theorem}
\begin{proof}
	Denote \(P_u = P(\cdot \given S_0 =u)\), \(\tau = \min\{k\geq 1:\ S_k\leq 0\}\). Then, note that \((S_k)_{k\geq 1}\) is a martingale. Thus, by Doob's optional stopping theorem,
	\begin{equation*}
		E_u(S_{\tau \wedge n}) = u.
	\end{equation*}
	But on the other hand,
	\begin{align*}
		E_u(S_{\tau \wedge n})
		&=
		E_u(\mathds{1}_{\tau> n} S_n) + E_u(\mathds{1}_{\tau\leq n} S_{\tau})
		\\
		&\begin{cases}
			\leq P_u(\tau > n) E_u(S_n\given \tau >n) \\
			\geq P_u(\tau > n) \sqrt{n} P_u(S_n\geq \sqrt{n}) -K
		\end{cases},
	\end{align*}
	as \(S_{\tau}\leq 0\) a.s., \(|S_{\tau} |\leq |X_{\tau}|\leq K\), and by FKG inequality we have
	\begin{equation*}
		E_u(\mathds{1}_{\tau> n} S_n)
		\geq
		P_u(\tau >n) E_u\bigl(\max(S_n,0)\bigr)
		\geq
		P_u(\tau >n) \sqrt{n} P_0(S_n \geq \sqrt{n}),
	\end{equation*}
	(both \(\tau\) and \(\max(S_n,0)\) are non-decreasing functions of \(X_1,\dots,X_n\)). We can now use that by the (inhomogeneous) CLT (Theorem~\ref{thm:inhomog_CLT}) and the uniform lower bound on the variances (provided by \(\alpha^2>0\)), \(P_0(S_n \geq \sqrt{n})\geq c\) for some \(c>0\) and \(n\) large enough (uniformly over the sequence \(X_1,X_2,\dots\)). This yields
	\begin{equation*}
		P_u(\tau > n) \sqrt{n} c -K \leq u.
	\end{equation*}
	Rearranging gives the upper bound. For the lower bound, we claim that
	\begin{equation}
		\label{eq:prf:martingal_bnd}
		E_u(S_n\given \tau >n) \leq c\sqrt{n}
	\end{equation}
	for some \(c>0\), and any \(n\) large enough (uniformly over the sequence \(X_1,X_2,\dots\)). Plugging this in the optional stopping display, we get
	\begin{equation*}
		P_u(\tau >n) \geq \frac{c u}{\sqrt{n}}.
	\end{equation*}
	Now, for \(u=0\), we have
	\begin{equation*}
		P_0(\tau>n)
		=
		E_0\bigl(\mathds{1}_{X_1 >0}P_{X_1}(\tau >n-1)\bigr)
		\geq
		\frac{c}{\sqrt{n}}E_0\bigl(\mathds{1}_{X_1 >0}X_1\bigr)
		\geq
		\frac{c\alpha}{\sqrt{n}}.
	\end{equation*}
	Remains to show~\eqref{eq:prf:martingal_bnd}. Introduce \(\tau' = \min\{k\geq 1: S_k\geq A\sqrt{n}\}\), with \(A>0\) to be fixed later.
	\begin{equation*}
		E_u(\mathds{1}_{\tau> n} S_n)
		=
		E_u(\mathds{1}_{\tau> n}\mathds{1}_{\tau'> n} S_n) + E_u(\mathds{1}_{\tau> n}\mathds{1}_{\tau'\leq n} S_n).
	\end{equation*}
	Now, \(E_u(\mathds{1}_{\tau> n}\mathds{1}_{\tau'> n} S_n)\leq P(\tau >n)A\sqrt{n}\). Remains to bound the other term. It is equal to
	\begin{equation*}
		P_u(\tau>n)\sum_{k=1}^n E_u(\mathds{1}_{\tau'=k} S_n\given \tau> n).
	\end{equation*}
	Using Markov's property and \(S_{i,j} = \sum_{k=i}^j X_k\),
	\begin{align*}
		&E_u(\mathds{1}_{\tau'=k} S_n\given \tau> n)
		\\
		&\quad=
		E_u\bigl( \mathds{1}_{S_1,\dots,S_{k-1}<A\sqrt{n}}\mathds{1}_{S_{k}>A\sqrt{n}} (S_{k-1} + X_k + E(S_{k+1,n} \given S_k, \tau >n) )\bgiven \tau >n\bigr)
		\\
		&\quad\leq
		E_u\bigl( \mathds{1}_{\tau' = k} (A\sqrt{n} + K + E(S_{k+1,n} \given S_k, \tau >n) )\bgiven \tau >n\bigr).
	\end{align*}
	Now, taking \(A = \sqrt{2} K\), for any \(v>A\sqrt{n}\),
	\begin{multline*}
		E(S_{k+1,n} \given S_k=v, \tau >n)
		= \frac{E(S_{k+1,n} \mathds{1}_{\tau>n} \given S_k=v,\tau>k)}{P(\tau > n \given S_k = v, \tau >k)} \\
		\leq
		\frac{E(|S_{k+1,n}|)}{P(\tau > n \given S_k = v, \tau >k)}
		\leq
		2K\sqrt{n},
	\end{multline*}
	as
	\begin{multline*}
		P(\tau > n \given S_k = v, \tau >k)
		\geq
		P(\max_{i=k+1,\dots, n} |S_{k+1,i}|\leq A\sqrt{n})
		\\
		\geq
		1-\frac{E(S_{k+1,n}^2)}{A^2n}
		\geq
		1-\frac{K^2}{A^2}
		= \frac{1}{2},
	\end{multline*}
	by Doob's submartingale inequality. Combining everything gives~\eqref{eq:prf:martingal_bnd}.
\end{proof}

\bibliographystyle{plain}
\bibliography{BibTeX}

\end{document}